\documentclass[11pt,a4paper]{article}
\usepackage[a4paper]{geometry}
\usepackage{amssymb,latexsym,amsmath,amsfonts,amsthm}
\usepackage{graphicx}
\usepackage{epstopdf}
\usepackage{float,amsmath,amssymb,mathrsfs,bm,multirow,graphics}
\usepackage{tikz}
\usepackage{pgflibraryshapes}
\usetikzlibrary{arrows,decorations.markings}
\usepackage{subfigure}
\usepackage{overpic}
\usepackage{fullpage}
\usepackage{comment,verbatim}
\usepackage{hyperref}
\usepackage{color}
\usepackage{mathrsfs}
\usepackage[title,titletoc]{appendix}
\usepackage{ulem}
\usepackage{enumitem}

\DeclareMathOperator{\diag}{diag}
\DeclareMathOperator*{\Res}{Res}

\DeclareMathOperator\tr{{tr}}

\renewcommand{\Re}{\mathrm{Re}\,}
\renewcommand{\Im}{\mathrm{Im}\,}

\newcommand{\ud}{\,\mathrm{d}}

\newcommand{\msf}{\mathsf}
\newcommand{\ii}{\mathrm{i}}

\newcommand{\Boh}{\mathcal{O}}

\newtheorem{theorem}{Theorem}[section]
\newtheorem{lemma}[theorem]{Lemma}
\newtheorem{proposition}[theorem]{Proposition}

\newtheorem{corollary}[theorem]{Corollary}

\newtheorem{rhp}[theorem]{RH problem}

\theoremstyle{definition}

\theoremstyle{remark}

\numberwithin{equation}{section}

\begin{document}
\title{Asymptotics of the hard edge Pearcey determinant}
\author{Luming Yao\footnotemark[1] ~and ~Lun Zhang\footnotemark[2]}

\renewcommand{\thefootnote}{\fnsymbol{footnote}}
\footnotetext[1]{School of Mathematical Sciences, Fudan University, Shanghai 200433, China. E-mail: \texttt{lumingyao@fudan.edu.cn}}
\footnotetext[2]{School of Mathematical Sciences and Shanghai Key Laboratory for Contemporary Applied Mathematics, Fudan University, Shanghai 200433, China. E-mail: \texttt{lunzhang@fudan.edu.cn}}
\date{\today}
\maketitle

\begin{abstract}
We study the Fredholm determinant of an integral operator associated to the hard edge Pearcey kernel. This determinant appears in a variety of random matrix and non-intersecting paths models. By relating the logarithmic derivatives of the Fredholm determinant to a $3 \times 3$ Riemann-Hilbert problem, we obtain asymptotics of the determinant, which is also known as the large gap asymptotics for the corresponding point process.

%$\det(I - \mathcal K_{s})$, where $\mathcal K_{s}$ is the trace class operator acting on $L^2(0, s)$ with the hard edge Pearcey kernel appearing at the cusp of non-intersecting squared Bessel paths model. We derive the asymptotics of the Fredholm determinant as $s \to \infty$, which is also interpreted as the large gap probability in random matrix theory. Our method is based on the Deift–Zhou nonlinear steepest descent analysis for a $3 \times 3$ Riemann-Hilbert problem.
\end{abstract}

\tableofcontents

\section{Introduction and statement of the result}
In a classical work \cite{Dyson}, Dyson observed that eigenvalues of the process version of Gaussian unitary ensemble share the same statistics with non-intersecting Brownian motions. Since then, one dimensional Markov processes conditioned not to intersect have played an important role in the studies of random matrix theory and a variety of problems arising from probability and mathematical physics. An important motivation behind is that these models give rise to universal determinantal point processes, which also appear in a wide range of interacting particle systems.

The hard edge Pearcy process is a concrete example related to a model of non-intersecting squared Bessel paths. The squared Bessel process is a diffusion process depending on a parameter $\alpha>-1$ with transition probability function constructed via the the modified Bessel functions of the first kind; cf. \cite{BS}. If $d = 2(\alpha+1)$ is an integer, it can be obtained as the square of the distance to the origin of a $d$-dimensional Brownian motion. The model consists of $n$ independent copies of the squared Bessel process such that they all start at some fixed positions at $t=0$, end at some fixed positions at $t=T$, and do not intersect one another for $0<t<T$. By \cite{Konig01}, non-intersecting squared Bessel paths provides a process version of the Laguerre unitary ensemble, and different types of initial and ending conditions are considered in \cite{Del13,DKRZ12,Kat12,KIK08,KatTan11,KatTan04,Kuijlaars2011,KMW09}. If all the paths start at the same positive value when $t=0$ and end at $x=0$ when $t=T$, it comes out that as $n\to\infty$, after proper scaling, the paths will fill in a region in the $tx$-plane; see Figure \ref{fig:nsbp} for an illustration.
\begin{figure}[t]
 \centering
  \includegraphics[scale=.55]{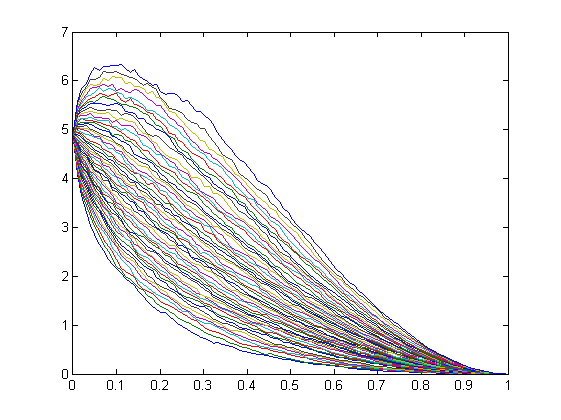}
  \caption{Simulation picture of 50 rescaled non-intersecting squared Bessel paths with $\alpha=4$ that start at $x=5$ and end at $x=0$.}
  \label{fig:nsbp}
\end{figure}
It is readily seen from the numerical simulation that there is a critical time such that the lowest path stays away from the hard edge at $x=0$ for any earlier time while stays close to $0$ for any later time. The local statistics are governed by classical Airy and Bessel processes from random matrix theory respectively; see \cite{KMW09}. After scaling around the critical time, one encounters a determinantal point process characterized by the following kernel (see \cite[Equations (1.19) and (1.23)]{Kuijlaars2011}):
\begin{align}\label{kernel}
&K_{\alpha}(x, y; \rho)
\nonumber
\\
& =\frac{1}{ (2 \pi \ii)^2}\int_{t\in \Gamma}\int_{s\in \Sigma}\frac{e^{\rho/t+1/(2t^2)-\rho/s -1/(2s^2)+xt-ys}}{s-t}
\left(\frac{t}{s}\right)^{\alpha}\ud t \ud s
\nonumber
\\
&= \frac{\mathcal{P}(x)\left[\mathcal{Q}''(y) - (\alpha - 2) \mathcal{Q}'(y) - \rho \mathcal{Q}(y)\right] - \mathcal{P}'(x) \left[y \mathcal{Q}'(y) - (\alpha - 1) \mathcal{Q}(y)\right] + y \mathcal{P}''(x) \mathcal{Q}(y)}{2 \pi \ii (x-y)}
\end{align}
for $x, y > 0$, where the parameters $\alpha > 1$, $\rho \in \mathbb{R}$,
\begin{equation}\label{def-PQ}
	\mathcal{P}(x) = \int_{\Gamma} t^{\alpha-3} e^{xt + \frac{\rho}{t} + \frac{1}{2t^2}} \ud t, \qquad \mathcal{Q}(y) = \int_{\Sigma} t^{\alpha-4} e^{-yt - \frac{\rho}{t} - \frac{1}{2t^2}} \ud t,
\end{equation}
and the contours $\Gamma$ and $\Sigma$ are illustrated in Figure \ref{figure-contour}. The functions $\mathcal{P}$ and $\mathcal{Q}$ in \eqref{def-PQ} satisfy the third order ordinary differential equations
\begin{align}
x \mathcal{P}'''(x) + \alpha \mathcal{P}''(x) - \rho \mathcal{P}'(x) - \mathcal{P}(x) &= 0,\label{3rd-P}\\
y \mathcal{Q}''(y) + (3 - \alpha) \mathcal{Q}''(y) - \rho \mathcal{Q}'(y) + \mathcal{Q}(y) &= 0,
\end{align}
respectively. Following the terminology in \cite{DV15}, we call $K_\alpha$ the hard edge Pearcey kernel, as it appears at the cusp of non-intersecting squared Bessel paths model.

It was expected in \cite{Kuijlaars2011} that $K_\alpha$ also admits an alternative representation in terms of the Bessel functions of the first kind, which was derived earlier by Desrosiers and Forrester in the context of perturbed chiral Gaussian unitary ensemble \cite{DF08}. This conjecture was later resolved in \cite{DV15}. The universal feature of hard edge Pearcey process can be seen from its appearances in the investigation of subjects as diverse as Jacobi growth process \cite{CerKuan20}, non-intersecting Brownian motions with walls \cite{LieWang17}, random surface growth models \cite{BK10,Cer15}, etc.

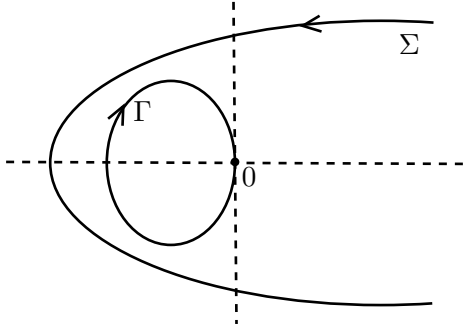
\begin{figure}[t]
\center
\tikzset{every picture/.style={line width=1pt}} %set default line width to 0.75pt

\begin{tikzpicture}[x=0.75pt,y=0.75pt,yscale=-1,xscale=1]

%Straight Lines [id:da7443919680345368]
\draw [line width=1][dash pattern={on 3pt off 3pt}]    (148,119) -- (378.5,120) ;
%Straight Lines [id:da8393515825439454]
\draw   [line width=1][dash pattern={on 3pt off 3pt}]  (261.5,38.5) -- (263,200.5) ;
%Shape: Circle [id:dp8840316595604266]
\draw  [line width=1][fill={rgb, 255:red, 0; green, 0; blue, 0 }  ,fill opacity=1 ] (260.75,119) .. controls (260.75,118.17) and (261.42,117.5) .. (262.25,117.5) .. controls (263.08,117.5) and (263.75,118.17) .. (263.75,119) .. controls (263.75,119.83) and (263.08,120.5) .. (262.25,120.5) .. controls (261.42,120.5) and (260.75,119.83) .. (260.75,119) -- cycle ;
%Shape: Ellipse [id:dp756194057367966]
\draw   [line width=1](198.25,119.5) .. controls (198.25,96.72) and (212.58,78.25) .. (230.25,78.25) .. controls (247.92,78.25) and (262.25,96.72) .. (262.25,119.5) .. controls (262.25,142.28) and (247.92,160.75) .. (230.25,160.75) .. controls (212.58,160.75) and (198.25,142.28) .. (198.25,119.5) -- cycle ;
\draw   (198.92,97.27) -- (206.49,90.96) -- (206.1,100.8) ;
%Shape: Arc [id:dp2369674192417811]
\draw  [draw opacity=0] (360.71,190.14) .. controls (352.33,190.71) and (343.75,191) .. (335,191) .. controls (243.87,191) and (170,158.99) .. (170,119.5) .. controls (170,80.01) and (243.87,48) .. (335,48) .. controls (343.96,48) and (352.76,48.31) .. (361.34,48.91) -- (335,119.5) -- cycle ; \draw   (360.71,190.14) .. controls (352.33,190.71) and (343.75,191) .. (335,191) .. controls (243.87,191) and (170,158.99) .. (170,119.5) .. controls (170,80.01) and (243.87,48) .. (335,48) .. controls (343.96,48) and (352.76,48.31) .. (361.34,48.91) ;
\draw   (305.45,53.42) -- (296.03,50.55) -- (304.48,45.48) ;

% Text Node
\draw (264,120) node [anchor=north west][inner sep=0.75pt]   [align=left] {$0$};
% Text Node
\draw (342.34,53) node [anchor=north west][inner sep=0.75pt]   [align=left] {$\Sigma$};
% Text Node
\draw (210,87) node [anchor=north west][inner sep=0.75pt]   [align=left] {$\Gamma$};

\end{tikzpicture}
\caption{The contours $\Gamma$ and $\Sigma$ in the definitions of $\mathcal{P}$, $\mathcal{Q}$ and $K_\alpha$.}
\label{figure-contour}
\end{figure}

Let $\mathcal{K}_s$ be the integral operator acting on $L^2(0, s)$, $s \ge 0$, with the hard edge Pearcey kernel $K_\alpha$ given in \eqref{kernel}. Due to the determinantal structure, the Fredholm determinant $\det(I - \mathcal K_{s})$ can be interpreted as the gap probability for the hard edge Pearcey process over the interval $(0, s)$. Intensive studies of various Fredholm determinants arising from random matrix theory have exhibited their close connections with integrable systems and elegant forms of the large gap asymptotics. The relevant results can be found in \cite{DKMVZ01,Dyson76,Ehr06,JMU81,Kra04,Widom94} for the sine determinant, in \cite{BBD08,DIK08,TWAiry} for the Airy determinant, in \cite{DeiBes,Ehr10,TWBessel} for the Bessel determinant, in \cite{Advan07,BerCaf12,Bre98,DXZ22b,DXZ21} for the Pearcey determinant, among others. For the gap probability of the hard edge Pearcey process, it has been shown in \cite{DXZ22a} and \cite{Gir14} that $\det(I - \mathcal K_{s})$ can be connected to two different integrable systems, although the precise relationship is not clear yet. In addition, asymptotics of the deformed case, i.e., $\det(I - \gamma \mathcal K_{s})$, $0<\gamma<1$, is also obtained in \cite{DXZ22a}. This in turn gives us large gap asymptotics of the thinned process. We contribute to these developments by establishing  large gap asymptotics for the hard edge Pearcey process stated below.
\begin{theorem}\label{th1}
Let \begin{equation}\label{def-F}
	F(s; \rho) := \ln \det(I - \mathcal K_{s}).
\end{equation}
	As $s \to \infty$, we have
	\begin{align}\label{F1}
		F(s; \rho) = & -\frac{9}{2^{14/3}}s^{\frac{4}{3}}+\frac{\rho}{2} s + \frac{3 \alpha - \rho^2}{2^{7/3}} s^{\frac{2}{3}} - \frac{\alpha \rho}{2^{2/3}} s^{\frac{1}{3}} -\frac{12\alpha^2+1}{72} \ln{s}
\nonumber \\
		& + \frac{\rho^4}{108} + \frac{\alpha \rho^2}{6} + C+\Boh(s^{-\frac{1}{3}}),
		\end{align}
uniformly for $\rho$ in any compact subset of $\mathbb{R}$, where $C$ is an undetermined constant independent of $\rho$ and $s$.
\end{theorem}

Some remarks about the above theorem are the following. Our asymptotic formula supports the so-called Forrester-Chen-Eriksen-Tracy conjecture; cf. \cite{Chen95,For93}. Based on a Coulomb fluid approach, this conjecture claims that the probability $E(s)$ of emptiness over the interval $(0,s)$ behaves like $\exp(-\mu s^{2\kappa+2})$ for large positive $s$ with $\mu$ being some constant, provided the density of state $\eta(x)$ satisfies $\eta(x)\sim x^{\kappa}$ as $x\to 0$. The present case corresponds to $\kappa=-1/3$. Asymptoics of $\det(I - \gamma \mathcal K_{s})$ exhibit significantly different asymptotic behaviours for $\gamma=1$ and $0<\gamma <1$. Indeed, by \cite[Theorem 2.2]{DXZ22a}, it follows that $\ln(\det(I - \gamma \mathcal K_{s}))\sim \Boh(s^{-2/3})$. Finally, we cannot evaluate explicitly the constant $C$ in \eqref{F1} with our method, which in general is  a challenging task; cf. \cite{Kra09}.

The proof of Theorem \ref{th1} relies on the integrable structure of hard edge Pearcey kernel. This special structure enables us to relate various derivatives of $F$ to a $3 \times 3$ Riemann-Hilbert (RH) problem under the general framework \cite{BD2002,Dei97}. In Section \ref{sec:preli}, we recall this RH problem derived in \cite{DXZ22a} and further establish its connection with $\partial F / \partial \rho$.
A key step in our analysis is the construction of $\lambda$-functions defined on a Riemann surface with a specified sheet structure, which is given in Section \ref{sec:lambdafunctions}. With the aid of these auxiliary functions, we perform a Deift-Zhou steepest descent analysis
on the relevant RH problem as $s\to \infty$ in Section \ref{sec:asyanaly}. The proof of Theorem \ref{th1} is an outcome of our asymptotic analysis, which is presented in Section \ref{sec:proof}.

\paragraph{Notations} Throughout this paper, the following notations are frequently used.
\begin{itemize}
  \item If $A$ is a matrix, then $(A)_{ij}$ stands for its $(i,j)$-th entry and $A^{\msf T}$ stands for its transposition. An unimportant entry of $A$ is denoted by $\ast$. We use $I$ to denote the identity matrix, and the size might differ in different contexts.
  \item We denote by $D(z_0; r)$ the open disc centred at $z_0$ with radius $r > 0$, i.e.,
  \begin{equation}\label{def:dz0r}
   D(z_0; r) := \{ z\in \mathbb{C} \mid |z-z_0|<r \},
   \end{equation}
  and denote by $\partial D(z_0, r)$ its boundary.
  \item As usual, the three Pauli matrices $\{\sigma_j\}_{j=1}^3$ are defined by
\begin{equation}\label{def:Pauli}
\sigma_1=\begin{pmatrix}
           0 & 1 \\
           1 & 0
        \end{pmatrix},
        \qquad
        \sigma_2=\begin{pmatrix}
        0 & -\ii \\
        \ii & 0
        \end{pmatrix},
        \qquad
        \sigma_3=
        \begin{pmatrix}
        1 & 0 \\
         0 & -1
         \end{pmatrix}.
\end{equation}

\end{itemize}

%-----------------------------------------------------------------------------------------------------------------------------------------------

\section{Preliminaries}\label{sec:preli}
It has been shown in \cite{DXZ22a} that $\partial F / \partial s$ is related to the local behavior of a $3 \times 3$ RH problem. In this section, we will recall the derivation of this RH problem and further establish its connection with $\partial F / \partial \rho$.
%\subsection{A Riemann-Hilbert characterization of the hard edge Pearcey kernel}

We start with a $3 \times 3$ RH problem which characterizes the hard edge Pearcey kernel $K_{\alpha}$, as given in \cite{Kuijlaars2011} and stated next.

\begin{rhp}\label{rhp-Psi}
\hfill
\begin{itemize}
\item[\rm (a)] $\Psi(z)=\Psi(z;\rho,\alpha)$ is analytic in $\mathbb{C} \setminus \Sigma_{\Psi}$,
where $\alpha>-1$ and $\rho$ are real parameters,
\begin{equation}
\Sigma_{\Psi}:=\cup_{k=0}^5 \Sigma_k \cup \{0\},
\end{equation}
 with
    \begin{equation}
    \Sigma_0= (0,\infty),\qquad  \Sigma_1=e^{\frac{\pi}{4}\ii}(0,\infty), \qquad  \Sigma_2=e^{\frac{3\pi}{4}\ii}(0,\infty), \quad
    \end{equation}
    and
    \begin{equation}
    \Sigma_{3+k}=-\Sigma_k, \quad k=0,1,2;
    \end{equation}
see Figure \ref{figure-Psi} for an illustration.

\begin{figure}[h]
\centering

\begin{tikzpicture}[x=0.75pt,y=0.75pt,yscale=-1,xscale=1]

%Straight Lines [id:da3429112897462063]
\draw   [line width=1] (139.53,160.53) -- (361.53,160.53) ;
%Straight Lines [id:da2403611060462012]
\draw    [line width=1](170.08,89.6) -- (330.98,231.47) ;
%Straight Lines [id:da7618984273675591]
\draw    [line width=1](169.86,233.16) -- (331.21,87.91) ;
\draw  [line width=1]  (318,163.67) -- (328,160.67) -- (318,157.67) ;
\draw  [line width=1]  (171,163.67) -- (181,160.67) -- (171,157.67) ;
\draw  [line width=1]  (302.59,118.32) -- (307.54,109.13) -- (298.34,114.08) ;
\draw  [line width=1]  (192.59,217.32) -- (197.54,208.13) -- (188.34,213.08) ;
\draw  [line width=1] (190.68,111.59) -- (199.87,116.54) -- (194.92,107.34) ;
\draw  [line width=1]  (298.68,206.59) -- (307.87,211.54) -- (302.92,202.34) ;
%Shape: Ellipse [id:dp837965754345402]
\draw  [fill={rgb, 255:red, 0; green, 0; blue, 0 }  ,fill opacity=1 ][line width=1.5]  (250.42,161.03) .. controls (250.42,160.76) and (250.47,160.53) .. (250.53,160.53) .. controls (250.6,160.53) and (250.65,160.76) .. (250.65,161.03) .. controls (250.65,161.31) and (250.6,161.53) .. (250.53,161.53) .. controls (250.47,161.53) and (250.42,161.31) .. (250.42,161.03) -- cycle ;

% Text Node
\draw (245,164.78) node [anchor=north west][inner sep=0.75pt]   [align=left] {0};
% Text Node
\draw (148,77.78) node [anchor=north west][inner sep=0.75pt]   [align=left] {$\Sigma_2$};
% Text Node
\draw (118,151.78) node [anchor=north west][inner sep=0.75pt]   [align=left] {$\Sigma_3$};
% Text Node
\draw (148,226.78) node [anchor=north west][inner sep=0.75pt]   [align=left] {$\Sigma_4$};
% Text Node
\draw (334,226.78) node [anchor=north west][inner sep=0.75pt]   [align=left] {$\Sigma_5$};
% Text Node
\draw (334,77.78) node [anchor=north west][inner sep=0.75pt]   [align=left] {$\Sigma_1$};
% Text Node
\draw (365,151.78) node [anchor=north west][inner sep=0.75pt]   [align=left] {$\Sigma_0$};
\end{tikzpicture}
 \caption{The jump contours $\Sigma_k$, $k=0,1,\ldots,5$, in the RH problem for $\Psi$.}
 \label{figure-Psi}
\end{figure}
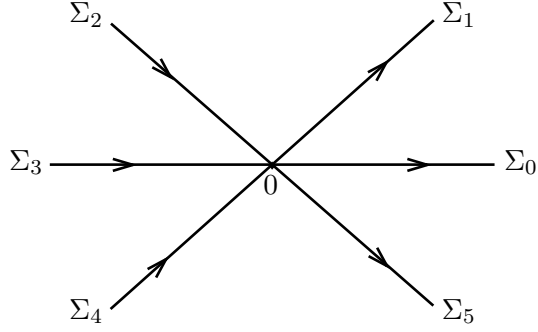

\item [\rm (b)] For $z \in \Sigma_{k}$, $k=0,1,\ldots,5$, $\Psi$ has continuous boundary values $\Psi_{\pm}(z)$, where the $+/-$-side of $\Sigma_k$ is the side which lies on the left/right of $\Sigma_k$, when traversing $\Sigma_k$ according to its orientation. These boundary values satisfy
    \begin{equation}
    \Psi_{+}(z)=\Psi_{-}(z) J_{\Psi}(z), \quad z\in \cup_{k=0}^5 \Sigma_k,
    \end{equation}
where
\begin{equation}\label{jump-Psi}
	J_{\Psi}(z):= \begin{cases}
	\begin{pmatrix}
	    0 & 1 & 0\\
	    -1 & 0 & 0\\
	    0 & 0 & 1
	\end{pmatrix}, &\quad z \in \Sigma_0,\\
	\begin{pmatrix}
	     1 & 0 & 0\\
	     1 & 1 & 0\\
	     0 & 0 & 1
	\end{pmatrix}, &\quad z \in \Sigma_1,\\
	\begin{pmatrix}
	     1 & 0 & 0\\
	     0 & 1 & e^{\alpha \pi \ii}\\
	     0 & 0 & 1
	\end{pmatrix}, &\quad z \in \Sigma_2,\\
	\begin{pmatrix}
	     1 & 0 & 0\\
	     0 & 0 & -e^{-\alpha \pi \ii}\\
	     0 & e^{-\alpha \pi \ii} & 0
	\end{pmatrix}, &\quad z \in \Sigma_3,\\
	\begin{pmatrix}
	     1 & 0 & 0\\
	     0 & 1 & e^{-\alpha \pi \ii}\\
	     0 & 0 & 1
	\end{pmatrix}, &\quad z \in \Sigma_4,\\
	\begin{pmatrix}
	     1 & 0 & 0\\
	     1 & 1 & 0\\
	     0 & 0 & 1
	\end{pmatrix}, &\quad z \in \Sigma_5.
	\end{cases}
\end{equation}
\item [\rm (c)] As $z \to \infty$ with $z\in \mathbb{C}\setminus \Sigma_{\Psi}$, we have
\begin{multline}\label{infty-Psi}
    \Psi(z) = \frac{\ii z^{-\alpha/3}}{\sqrt{3}} \Psi_{0}\left(I + \frac{\Psi_1}{z}+\Boh(z^{-2})\right)\diag{(z^{\frac{1}{3}}, 1, z^{-\frac{1}{3}})}
  \\
  \times L_{\pm} \diag{(e^{\pm \frac{\alpha \pi }{3}\ii}, e^{\mp \frac{\alpha \pi }{3}\ii}, 1)} e^{\Theta (z)}, \quad \pm \Im{z} > 0,
\end{multline}
where
\begin{equation}\label{def-Psi-0}
	\Psi_0  = \begin{pmatrix}
         1 & \pi_3(\rho) & \pi_{6}(\rho)\\
         0 & 1 & \pi_3(\rho) + \rho/3\\
         0 & 0 & 1
    \end{pmatrix}, \qquad
    \Psi_1  = \begin{pmatrix}
         \ast & \ast & \ast\\
         \ast & \ast & \ast\\
         \pi_3(\rho) + 2\rho/3 & \ast & \ast
    \end{pmatrix}
\end{equation}
with
\begin{align}
     \pi_3(\rho) & = \frac{\rho (\rho^2+9\alpha-18)}{27},  \label{def-pi-3}\\
     \pi_6(\rho) & = \frac{\rho^6+(18\alpha-45)\rho^4 + (81\alpha^2-405\alpha+405)\rho^2-243\alpha^2+729\alpha-405}{2\cdot3^6}, \label{def-pi-6}
\end{align}
%and $\ast$ stands for some unimportant entry.
\begin{equation}\label{def-L+-}
    L_+ = \begin{pmatrix}
         \omega & \omega^2 & 1\\
         1 & 1 & 1\\
         \omega^2 & \omega & 1
    \end{pmatrix}, \qquad
    L_- = \begin{pmatrix}
         \omega^2 & -\omega & 1\\
         1 & -1 & 1\\
         \omega & -\omega^2 & 1
    \end{pmatrix},
\end{equation}
with $\omega = e^{2\pi \ii/3}$, and
\begin{equation}\label{def-theta}
    \Theta (z) = \Theta (z; \rho) =\begin{cases}
    \diag{(\theta_1(z; \rho), \theta_2(z; \rho), \theta_3(z; \rho))}, & \quad \Im{z} >0,\\
    \diag{(\theta_2(z; \rho), \theta_1(z; \rho), \theta_3(z; \rho))}, & \quad \Im{z} <0,
    \end{cases}
\end{equation}
with
\begin{equation}\label{def-theta-k}
    \theta_k(z; \rho) = \frac{3}{2} \omega^{2k} z^{\frac{2}{3}} + \rho \omega^k z^{\frac{1}{3}}, \quad k = 1, 2, 3.
\end{equation}

\item [\rm (d)] As $z \to 0$, we have
\begin{align}
    \Psi(z) \begin{pmatrix}
         z^{\alpha} & 0 & 0\\
         0 & z^{\alpha} & 0\\
         0 & 0 & 1
    \end{pmatrix} & = \Boh(1), \quad \quad 0 < |\arg z| < \frac{\pi}{4},\\
    \Psi(z) \begin{pmatrix}
         1 & 0 & 0\\
         0 & z^{\alpha} & 0\\
         0 & 0 & 1
    \end{pmatrix} & = \Boh(1), \quad \quad \frac{\pi}{4} < |\arg z| < \frac{3\pi}{4},\\
    \Psi(z) \begin{pmatrix}
         1 & 0 & 0\\
         0 & z^{\alpha} & 0\\
         0 & 0 & z^{\alpha}
    \end{pmatrix} & = \Boh(1), \quad \quad \frac{3\pi}{4} < |\arg z| < \pi.
\end{align}
\end{itemize}
\end{rhp}

By \cite[Theorem 1.4 and Proposition 5.2]{Kuijlaars2011}, RH problem \ref{rhp-Psi} for $\Psi$ has a unique solution which can be constructed through the functions
\begin{equation}\label{integral-p}
\mathcal {P}_k(z) :=
\begin{cases}
\int_{\gamma_k} t^{\alpha-3}e^{zt+\frac{\rho}{t}+\frac{1}{2t^2}} \ud t, & \quad \textrm{$k = 1$ with $-\frac{\pi}{2}< \arg t < \frac{\pi}{2}$,}\\
e^{-\alpha \pi \ii} \int_{\gamma_k} t^{\alpha-3}e^{zt+\frac{\rho}{t}+\frac{1}{2t^2}} \ud t, & \quad \textrm{$k = 2$ with $\frac{\pi}{2}< \arg t < \frac{3\pi}{2}$,}\\
e^{-\alpha \pi \ii} \int_{\gamma_k} t^{\alpha-3}e^{zt+\frac{\rho}{t}+\frac{1}{2t^2}} \ud t, & \quad \textrm{$k = 3$ with $0 < \arg t < \pi $,}\\
e^{\alpha \pi \ii} \int_{\gamma_k} t^{\alpha-3}e^{zt+\frac{\rho}{t}+\frac{1}{2t^2}} \ud t, & \quad \textrm{$k = 4$ with $-\pi < \arg t < 0$,}
\end{cases}
\end{equation}
where the contours $\gamma_k$, $k=1,\ldots,4$, are illustrated in Figure \ref{fig:gamma-k}. We refer to \cite{Kuijlaars2011} for the precise descriptions of the contours $\gamma_k$ and the construction of $\Psi$. It is worthwhile to mention that $\mathcal {P}_k$, $k=1,2,3,4$, satisfies the differential equation \eqref{3rd-P} and any three of them are linearly independent.

%The contours $\gamma_1$ and $\gamma_2$ are symmetric with respect to the imaginary axis, which are tangent to the imaginary axis near the origin. The contours $\gamma_3$ and $\gamma_4$ start at infinity under the condition that $\Re (zt)<0$ as $t\to \infty$, and end  at the origin along the imaginary axis. It is readily seen that $\mathcal{P}_1(z)$ and $\mathcal{P}_2(z)$ are entire functions in $z$, while $\mathcal{P}_3(z)$ and $\mathcal{P}_4(z)$ are analytic for $z \in \C \setminus \ii \R_-$ and $z \in \C \setminus \ii \R_+$, respectively.

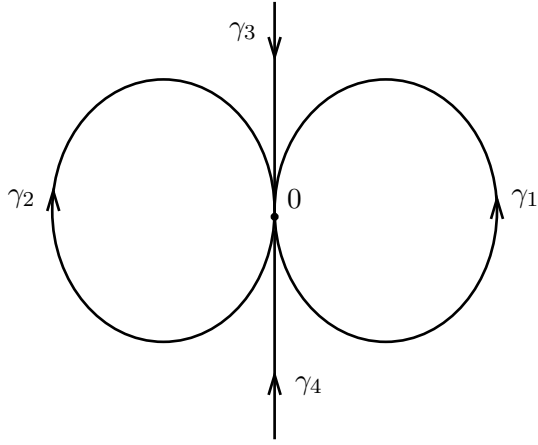
\begin{figure}[h]
\centering

\begin{tikzpicture}[x=0.75pt,y=0.75pt,yscale=-1,xscale=1]

%Shape: Ellipse [id:dp35021102445158725]
\draw  [line width=1]  (215,147) .. controls (215,110.55) and (239.85,81) .. (270.5,81) .. controls (301.15,81) and (326,110.55) .. (326,147) .. controls (326,183.45) and (301.15,213) .. (270.5,213) .. controls (239.85,213) and (215,183.45) .. (215,147) -- cycle ;
%Shape: Ellipse [id:dp22995339293270667]
\draw  [line width=1]  (326,147) .. controls (326,110.55) and (350.85,81) .. (381.5,81) .. controls (412.15,81) and (437,110.55) .. (437,147) .. controls (437,183.45) and (412.15,213) .. (381.5,213) .. controls (350.85,213) and (326,183.45) .. (326,147) -- cycle ;
%Straight Lines [id:da1556311106102508]
\draw [line width=1]    (326,42) -- (326,147) ;
%Straight Lines [id:da8695535691222038]
\draw [line width=1]    (326,157) -- (326,262) ;
\draw  [line width=1]  (329,59.67) -- (326,69.67) -- (323,59.67) ;
\draw  [line width=1]  (440,151.17) -- (437,141.17) -- (434,151.17) ;
\draw  [line width=1]  (329,239.67) -- (326,229.67) -- (323,239.67) ;
\draw  [line width=1]  (218.5,147.17) -- (215.5,137.17) -- (212.5,147.17) ;
%Shape: Circle [id:dp4127277343627058]
\draw  [fill={rgb, 255:red, 0; green, 0; blue, 0 }  ,fill opacity=1 ][line width=1.5]  (327,150.08) .. controls (327,149.53) and (326.55,149.08) .. (326,149.08) .. controls (325.45,149.08) and (325,149.53) .. (325,150.08) .. controls (325,150.64) and (325.45,151.08) .. (326,151.08) .. controls (326.55,151.08) and (327,150.64) .. (327,150.08) -- cycle ;

% Text Node
\draw (191.5,135) node [anchor=north west][inner sep=0.75pt]   [align=left] {$\gamma_{2}$};
% Text Node
\draw (335,229) node [anchor=north west][inner sep=0.75pt]   [align=left] {$\gamma_{4}$};
% Text Node
\draw (302,51) node [anchor=north west][inner sep=0.75pt]   [align=left] {$\gamma_{3}$};
% Text Node
\draw (443,135) node [anchor=north west][inner sep=0.75pt]   [align=left] {$\gamma_{1}$};
% Text Node
\draw (331,135) node [anchor=north west][inner sep=0.75pt]   [align=left] {0};
\end{tikzpicture}
\caption{The contour of integration $\Gamma_k$ in the definition of $\mathcal{P}_k(z)$, $k=1,2,3,4$.}
\label{fig:gamma-k}
\end{figure}

Define
\begin{equation} \label{eq: tilde-psi}
\widetilde\Psi(z)=  \frac{e^{\rho^2/6}}{\sqrt{2 \pi}}
\begin{pmatrix}
\mathcal{P}_2(z) & \mathcal{P}_3(z) & \mathcal{P}_1(z) \\
\mathcal{P}'_2(z) & \mathcal{P}'_3(z) & \mathcal{P}'_1(z)  \\
\mathcal{P}''_2(z) & \mathcal{P}''_3(z) & \mathcal{P}''_1(z)
\end{pmatrix}, \qquad z\in \mathbb{C}\setminus \ii \mathbb{R}_-.
\end{equation}
It is shown in \cite{Kuijlaars2011} that
\begin{equation}\label{eq:tildepsi}
\Psi(z)=\widetilde\Psi(z), \qquad \frac{\pi}{4} < \arg{z} < \frac{3\pi}{4},
\end{equation}
and the hard edge Pearcey kernel \eqref{kernel} admits the following representation in terms of $\widetilde{\Psi}$:
\begin{equation}\label{kernel1}
K_{\alpha} (x, y; \rho) =  \frac{1}{2 \pi \ii (x-y)} \begin{pmatrix}
0 & 1 & 0
\end{pmatrix}\widetilde{\Psi}(y)^{-1} \widetilde{\Psi}(x) \begin{pmatrix}
1\\
0\\
0
\end{pmatrix}, \quad x, y > 0.
\end{equation}

From \eqref{kernel1}, it is easily seen that
\begin{equation}\label{kernel2}
K_{\alpha} (x, y; \rho) = \frac{\mathbf{f}(x)^{\msf T} \mathbf{h}(y)}{x-y},
\end{equation}
where recall that the superscript $^\msf T$ denotes transpose operation,
\begin{equation}\label{def-fh}
\mathbf{f}(x) = \begin{pmatrix}
f_1(x)\\
f_2(x)\\
f_3(x)
\end{pmatrix} := \widetilde{\Psi}(x) \begin{pmatrix}
1\\
0\\
0
\end{pmatrix}, \quad \mathbf{h}(y) = \begin{pmatrix}
h_1(y)\\
h_2(y)\\
h_3(y)
\end{pmatrix} := \frac{1}{2 \pi \ii} \widetilde{\Psi}(y)^{-\msf T} \begin{pmatrix}
0\\
1\\
0
\end{pmatrix}.
\end{equation}
This integrable structure of $K_\alpha$ (in the sense of \cite{IIKS90}) particularly implies that the associated resolvent kernel is also integrable. Indeed, let $\msf R $ be the kernel of the resolvent operator $(I - \mathcal K_{s})^{-1}\mathcal K_{s}$. It then follows from \cite[Lemma 2.12]{Dei97} that
\begin{equation}\label{R}
\msf R(u, v) = \frac{\mathbf{F}(u)^{\msf T} \mathbf{H}(v)}{u-v},
\end{equation}
where
\begin{equation}\label{def-FH}
\mathbf{F}(u) = \begin{pmatrix}
F_1(u)
\\
F_2(u)
\\
F_3(u)
\end{pmatrix} := (I - \mathcal K_{s})^{-1} \mathbf{f}(u)=Y(u) \mathbf{f}(u)
\end{equation}
\begin{equation}
\mathbf{H}(v) = \begin{pmatrix}
H_1(v)\\
H_2(v)\\
H_3(v)
\end{pmatrix} := (I - \mathcal K_{s})^{-1} \mathbf{h}(v)= Y(v)^{- \msf T} \mathbf{h}(v),
\end{equation}
with
\begin{equation}\label{def-Y}
Y(z) = I - \int_0^s \frac{\mathbf{F}(t) \mathbf{h}(t)^{\msf T}}{t-z} \ud t.
\end{equation}
Moreover, $Y$ is the unique solution of the following RH problem.

\begin{rhp}\label{rhp:Y}
\hfill
\begin{itemize}
    \item [\rm (a)] $Y(z)$ is analytic in $\mathbb{C} \setminus [0, s]$.
    \item [\rm (b)] For $x \in (0, s)$, we have
    \begin{equation}\label{jump-Y}
        Y_+(x) = Y_-(x) (I-2 \pi \ii \mathbf{f}(x) \mathbf{h}(x)^{\msf T}),
    \end{equation}
    where the functions $\mathbf{f}$ and $\mathbf{h}$ are defined in \eqref{def-fh}.
    \item [\rm (c)] As $z \to \infty$, we have
\begin{equation}\label{infty-Y}
    Y(z) = I + \frac{Y_1}{z} + \Boh(z^{-2}),
\end{equation}
where the function $Y_1$ is independent of $z$.

    \item [\rm (d)] As $z \to 0$, we have
\begin{equation}\label{0-Y}
    Y(z) = \begin{cases}
    \Boh(z^{\alpha} \ln{z}), & \quad \alpha \in \mathbb{N} \cup \{0\},\\
    \Boh(z^{\alpha}), & \quad \alpha \notin \mathbb{Z}.
    \end{cases}
\end{equation}
    \item [\rm (e)] As $z \to s$, we have
    $
        Y(z) = \Boh(\ln{(z-s)}).
    $
\end{itemize}
\end{rhp}

The RH problem that is related to the partial derivatives of $F$ is then constructed by using the functions $\Psi$ and $Y$.
Let
\begin{equation}\label{def:sigmais}
\Sigma_0^{(s)}=(s,+\infty), \qquad \Sigma_1^{(s)}=s+e^{\frac{\pi}{4}\ii}(0,+\infty),\qquad \Sigma_5^{(s)}=s+e^{-\frac{\pi}{4}\ii}(0,+\infty),
\end{equation}
which are parallel to the rays $\Sigma_i$, $i=0,1,5$, respectively. Clearly, the rays $\Sigma_1^{(s)}$, $\Sigma_2$, $\Sigma_4$, $\Sigma_5^{(s)}$ and $\mathbb{R}$ are the boundaries of six regions $\msf D_i$, $i=1,\ldots,6$; see Figure \ref{figure-X} for an illustration.
We now define
\begin{equation}\label{def-X}
    X(z) = \begin{cases}
    Y(z) \widetilde{\Psi}(z), & \quad z \in \msf D_2, \\
    Y(z) \widetilde{\Psi}(z)
    \begin{pmatrix}
    1 & -1 & 0
    \\
    0 & 1 & 0
    \\
    0 & 0 & 1
    \end{pmatrix}, & \quad z \in \msf D_5, \\
    Y(z) \Psi(z), & \quad z \in \mathbb{C} \setminus \{\msf D_2 \cup \msf D_5\},
    \end{cases}
\end{equation}
where $\widetilde{\Psi}(z)$ is given in \eqref{eq: tilde-psi}. With the aid of RH problems \ref{rhp-Psi} and \ref{rhp:Y}, it is readily seen that $X$ solves the following RH problem (see \cite[Proposition 3.5]{DXZ22a}).

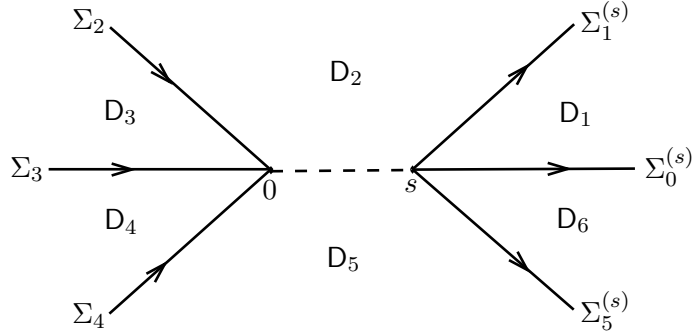
\begin{figure}[t]
\centering

\begin{tikzpicture}[x=0.75pt,y=0.75pt,yscale=-1,xscale=1]

%Straight Lines [id:da3429112897462063]
\draw    [line width=1](139.53,160.53) -- (250.53,160.53) ;
%Straight Lines [id:da2403611060462012]
\draw    [line width=1](170.2,89.1) -- (250.65,161.03) ;
%Straight Lines [id:da7618984273675591]
\draw   [line width=1] (169.86,233.16) -- (250.53,160.53) ;
\draw  [line width=1]  (171,163.67) -- (181,160.67) -- (171,157.67) ;
\draw  [line width=1]  (192.59,217.32) -- (197.54,208.13) -- (188.34,213.08) ;
\draw  [line width=1]  (190.68,111.59) -- (199.87,116.54) -- (194.92,107.34) ;
%Shape: Ellipse [id:dp837965754345402]
\draw  [fill={rgb, 255:red, 0; green, 0; blue, 0 }  ,fill opacity=1 ][line width=1.5]  (250.42,161.03) .. controls (250.42,160.76) and (250.47,160.53) .. (250.53,160.53) .. controls (250.6,160.53) and (250.65,160.76) .. (250.65,161.03) .. controls (250.65,161.31) and (250.6,161.53) .. (250.53,161.53) .. controls (250.47,161.53) and (250.42,161.31) .. (250.42,161.03) -- cycle ;
%Straight Lines [id:da8161230921495752]
\draw   [line width=1] (321.09,160.7) -- (432.2,160.2) ;
%Straight Lines [id:da6531892829408155]
\draw    [line width=1](321.31,160.7) -- (401.65,231.14) ;
%Straight Lines [id:da65139756403345]
\draw    [line width=1](321.09,160.7) -- (401.87,87.57) ;
\draw  [line width=1]  (388.67,163.33) -- (398.67,160.33) -- (388.67,157.33) ;
\draw  [line width=1]  (373.25,117.99) -- (378.2,108.8) -- (369.01,113.75) ;
\draw  [line width=1]  (369.34,206.25) -- (378.54,211.2) -- (373.59,202.01) ;
%Shape: Ellipse [id:dp7684203740072943]
\draw  [fill={rgb, 255:red, 0; green, 0; blue, 0 }  ,fill opacity=1 ][line width=1.5]  (321.09,160.7) .. controls (321.09,160.42) and (321.14,160.2) .. (321.2,160.2) .. controls (321.26,160.2) and (321.31,160.42) .. (321.31,160.7) .. controls (321.31,160.98) and (321.26,161.2) .. (321.2,161.2) .. controls (321.14,161.2) and (321.09,160.98) .. (321.09,160.7) -- cycle ;
%Straight Lines [id:da30291332412939753]
\draw  [line width=1][dash pattern={on 4.5pt off 4.5pt}]  (250.53,161.53) -- (321.55,160.91) ;

% Text Node
\draw (245,164.78) node [anchor=north west][inner sep=0.75pt]   [align=left] {0};
% Text Node
\draw (150,78.78) node [anchor=north west][inner sep=0.75pt]   [align=left] {$\Sigma_2$};
% Text Node
\draw (119,153) node [anchor=north west][inner sep=0.75pt]   [align=left] {$\Sigma_3$};
% Text Node
\draw (150,226.78) node [anchor=north west][inner sep=0.75pt]   [align=left] {$\Sigma_4$};
% Text Node
\draw (315.67,164.44) node [anchor=north west][inner sep=0.75pt]   [align=left] {$s$};
% Text Node
\draw (403.67,220.44) node [anchor=north west][inner sep=0.75pt]   [align=left] {$\Sigma_5^{(s)}$};
% Text Node
\draw (403.67,74.44) node [anchor=north west][inner sep=0.75pt]   [align=left] {$\Sigma_1^{(s)}$};
% Text Node
\draw (435.67,148) node [anchor=north west][inner sep=0.75pt]   [align=left] {$\Sigma_0^{(s)}$};

\draw(165.67,124.4) node [anchor=north west][inner sep=0.75pt]    {$\msf D_3$};
% Text Node
\draw (166,179.4) node [anchor=north west][inner sep=0.75pt]    {$\msf D_4$};
% Text Node
\draw (278.33,104.73) node [anchor=north west][inner sep=0.75pt]    {$\msf D_2$};
% Text Node
\draw (277,198.07) node [anchor=north west][inner sep=0.75pt]    {$\msf D_5$};
% Text Node
\draw (393,125.4) node [anchor=north west][inner sep=0.75pt]    {$\msf D_1$};
% Text Node
\draw (391.67,178.73) node [anchor=north west][inner sep=0.75pt]    {$\msf D_6$};
\end{tikzpicture}
\caption{Regions $\msf D_{i}$, $i=1,\ldots,6$, and the jump contours for the RH problem for $X$.}
\label{figure-X}
\end{figure}
\begin{rhp}\label{rhp-X}
\hfill
\begin{itemize}
    \item [\rm (a)] $X (z)$ is analytic in $\mathbb{C} \setminus \Sigma_{X}$,
    where
    \begin{equation}
    \Sigma_X:=\cup_{i=2,3,4} \Sigma_{i} \cup \{0\} \cup_{i=0,1,5} \Sigma_{i}^{(s)}\cup\{s\},
    \end{equation}
    see the solid lines in Figure \ref{figure-X}.
    \item [\rm (b)] For $z \in \Sigma_{X}\setminus \{0,s\}$, we have
    \begin{equation}
        X_+(z) = X_-(z) J_{X}(z),
    \end{equation}
    where
    \begin{equation}\label{jump-X}
        J_{X}(z) = \begin{cases}
        J_{\Psi}(z), &\quad z \in \cup_{i=2}^4\Sigma_i,\\
        \begin{pmatrix}
	    0 & 1 & 0\\
	    -1 & 0 & 0\\
	    0 & 0 & 1
	\end{pmatrix}, &\quad z \in \Sigma_0^{(s)},\\
	\begin{pmatrix}
	     1 & 0 & 0\\
	     1 & 1 & 0\\
	     0 & 0 & 1
	\end{pmatrix}, &\quad z \in \Sigma_1^{(s)},\\
	\begin{pmatrix}
	     1 & 0 & 0\\
	     1 & 1 & 0\\
	     0 & 0 & 1
	\end{pmatrix}, &\quad z \in \Sigma_5^{(s)},
        \end{cases}
    \end{equation}
    with $J_\Psi$ given in \eqref{jump-Psi}.

    \item [\rm (c)] As $z \to \infty$ with $z\in \mathbb{C}\setminus \Sigma_X$, we have
    \begin{multline}\label{infty-X}
        X(z)  = \frac{\ii z^{-\alpha/3}}{\sqrt{3}}\Psi_{0} \left(I + \frac{X_1}{z} + \Boh(z^{-2})\right)
        \diag{(z^{\frac{1}{3}}, 1, z^{-\frac{1}{3}})}
        \\
        \times L_{\pm} \diag{(e^{\pm \frac{\alpha \pi }{3}\ii}, e^{\mp \frac{\alpha \pi }{3}\ii}, 1)} e^{\Theta (z)}, \quad \pm \Im{z} > 0,
    \end{multline}
    where $\Psi_{0}$, $L_{\pm}$ and $\Theta (z)$ are given in \eqref{def-Psi-0}, \eqref{def-L+-} and \eqref{def-theta}, respectively, and
    \begin{equation}\label{def-X1}
    X_1 = \Psi_1+ \Psi_{0}^{-1} Y_1 \Psi_{0}
    \end{equation}
    with $\Psi_1$ and $Y_1$ given in \eqref{def-Psi-0} and \eqref{infty-Y}.
    \item [\rm (d)] As $z \to 0$, we have
    \begin{equation}\label{0-X}
    X(z)=\begin{cases}
            \Boh(z^{-\alpha}), & \quad \alpha>0,\\
            \Boh(\ln{z}), & \quad \alpha=0,\\
            \Boh(1), & \quad -1 < \alpha <0.
        \end{cases}
    \end{equation}
  %  \begin{equation}
%        X (z) = X^{(L)}(z) \begin{cases}
%            \begin{pmatrix}
%                 1 & 0 & 0\\
%                 0 & z^{-\alpha} & 0\\
%                 0 & \frac{e^{\alpha \pi i}}{2 \pi i}\ln{z} & 1
%            \end{pmatrix}, & \quad z \in II \cup V, \alpha \in \mathbb{N} \cup \{0\},\\
%            \begin{pmatrix}
%                 1 & 0 & 0\\
%                 0 & z^{-\alpha} & 0\\
%                 0 & \frac{1}{2\ii\sin{(\alpha \pi)}} & 1
%            \end{pmatrix}, & \quad z \in II \cup V, \alpha \notin \mathbb{Z},
%        \end{cases}
%    \end{equation}
%    where we take the principal branch for $\ln{z}$ and $z^{\alpha}$ and $X^{(L)}(z)$ is analytic at $z=0$ satisfying the expansion
%    \begin{equation}\label{def-X-L}
%        X^{(L)} (z) = X_0^{(0)}(s)\left(I + X_1^{(0)}(s)z + \Boh(z^2)\right), \quad z \to 0,
%    \end{equation}
%    for some functions $X_0^{(0)}(s)$ and $X_1^{(0)}(s)$. The local behavior of $X$ near $z=0$ in other regions can be determined through \eqref{0-X} and the jump condition \eqref{jump-X}.
    \item [\rm (e)]As $z \to s$, we have $X(z) = \Boh(\ln{(z-s)})$.
   % \begin{equation}\label{s-X}
%        X (z) = X^{(R)} (z) \begin{pmatrix}
%             1 & -\frac{1}{2 \pi i}\ln{(z-s)} & 0\\
%             0 & 1 & 0\\
%             0 & 0 & 1
%        \end{pmatrix} \begin{cases}
%        I, & \quad z \in II,\\
%        \begin{pmatrix}
%             1 & -1 & 0\\
%             0 & 1 & 0\\
%             0 & 0 & 1
%        \end{pmatrix}, & \quad z \in V,
%        \end{cases}
%    \end{equation}
%    where the principal branch is taken for $\ln{(z-s)}$, and $X^{(R)} (z)$ is analytic at $z=s$ satisfying the expansion
%    \begin{equation}\label{def-X-R}
%        X^{(R)} (z) = X_0^{(1)}(s)\left(I + X_1^{(1)}(s)(z-s) + \Boh(z-s)^2\right), \quad z \to s,
%    \end{equation}
%    for some functions $X_0^{(1)}(s)$ and $X_1^{(1)}(s)$. The local behavior of $X$ near $z=s$ in other regions can be determined through \eqref{0-X} and the jump condition \eqref{jump-X}.
\end{itemize}
\end{rhp}

The relationship between $X$ and $F$ is given in the following lemma through some differential identities.
\begin{lemma}\label{pro-de}
Let $F$ be the function defined in \eqref{def-F}. We have
\begin{align}\label{ds-F}
\frac{\partial}{\partial s} F(s; \rho) &= - \frac{1}{2 \pi \ii} \lim_{z \to s}\left(X(z)^{-1} X'(z)\right)_{21}, \qquad z \in  \msf D_2,
\\
\label{drho-F} \frac{\partial}{\partial \rho} F(s; \rho) & =  - (X_1)_{31}+ \frac{\rho (\rho^2 + 9\alpha)}{27},
\end{align}
where $X_1$ is given in \eqref{infty-X}.
%where $X_1^{(1)}(s)$ is given in \eqref{def-X-R}, and
%\begin{equation}
%\frac{d}{d \rho} F(s; \rho) = \frac{d}{d \rho} \ln{\det{(I - K_{s, \alpha, \rho})}} =  - (X_1)_{31}+ \frac{\rho (\rho^2 + 9\alpha)}{27},
%\end{equation}
\end{lemma}
To prove Lemma \ref{pro-de}, we need the following proposition.
%For later use, we need the following linear differential equations for $\Psi$ with respect to $z$ and $\rho$.
\begin{proposition}\label{rho-Psi}
Let $\Psi$ be the unique solution to the RH problem \ref{rhp-Psi}. We have
%\begin{equation}
%\frac{\partial \Psi}{\partial z} = \left[\begin{pmatrix}
%0 & 1 & 0\\
%0 & 0 & 1\\
%0 & 0 & 0
%\end{pmatrix} + \frac{1}{z}\begin{pmatrix}
%0 & 0 & 0\\
%0 & 0 & 0\\
%1 & \rho & -\alpha
%\end{pmatrix}\right]\Psi,
%\end{equation}
%and
\begin{equation}\label{partial-rho-Psi}
\frac{\partial \Psi}{\partial \rho} = \begin{pmatrix}
-2 \rho/3 & \alpha - 1 & z\\
1 & \rho/3 & 0 \\
0 & 1 & \rho/3
\end{pmatrix} \Psi.
\end{equation}
\end{proposition}
\begin{proof}
Since the jumps of $\Psi$ are constant matrices, by \eqref{eq:tildepsi}, it suffices to show \eqref{partial-rho-Psi} holds for $\widetilde \Psi$.
Recall that $\mathcal{P}_1$,  $\mathcal{P}_2$ and $\mathcal{P}_3$ are three linearly independent solutions of \eqref{3rd-P}, by differentiating both sides of \eqref{3rd-P} with respect to $\rho$, it follows that for $k = 1, 2, 3$,
\begin{equation}
\frac{\partial \mathcal{P}_k}{\partial \rho} =
z \frac{\partial \mathcal{P}'''_k}{\partial \rho}+\alpha \frac{\partial \mathcal{P}''_k}{\partial \rho}-\mathcal{P}'_k-\rho \frac{\partial \mathcal{P}'_k}{\partial \rho}.
\end{equation}
By \eqref{integral-p}, it is readily seen that
\begin{equation}
\frac{\partial \mathcal{P}'''_k}{\partial \rho} = \mathcal{P}''_k, \qquad \frac{\partial \mathcal{P}''_k}{\partial \rho} = \mathcal{P}'_k, \qquad
\frac{\partial \mathcal{P}'_k}{\partial \rho} = \mathcal{P}_k.
\end{equation}
Thus,
\begin{equation}
\frac{\partial \mathcal{P}_k}{\partial \rho} =
z \mathcal{P}_k'' + (\alpha-1) \mathcal{P}_k'  - \rho \mathcal{P}_k.
\end{equation}
A combination of the above two formulas and \eqref{eq: tilde-psi} gives us \eqref{partial-rho-Psi} for $\widetilde \Psi$.

%To show, we apply \eqref{3rd-P} and and afte a direct computation,
%\begin{equation}
%\begin{split}
%\frac{\partial \mathcal{P}_k}{\partial \rho} &= \begin{cases}
%\int_{\gamma_k} t^{\alpha-4}e^{zt+\frac{\rho}{t}+\frac{1}{2t^2}} \ud t, & \quad \textrm{$k = 1$},\\
%e^{-\alpha \pi \ii} \int_{\gamma_k} t^{\alpha-4}e^{zt+\frac{\rho}{t}+\frac{1}{2t^2}} \ud t, & \quad \textrm{$k = 2, 3$,}\\
%e^{\alpha \pi \ii} \int_{\gamma_k} t^{\alpha-4}e^{zt+\frac{\rho}{t}+\frac{1}{2t^2}} \ud t, & \quad \textrm{$k = 4$}
%\end{cases} \\
%&= z \mathcal{P}_k'' + \alpha \mathcal{P}_k' - \mathcal{P}_k' - \rho \mathcal{P}_k,
%\end{split}
%\end{equation}
%\begin{equation}
%\frac{\partial \mathcal{P}'_k}{\partial \rho} = \mathcal{P}_k,
%\end{equation}
%and
%\begin{equation}
%\frac{\partial \mathcal{P}''_k}{\partial \rho} = \mathcal{P}'_k,
%\end{equation}
%which leads to \eqref{partial-rho-Psi}.

This finishes the proof of Proposition \ref{rho-Psi}.
\end{proof}

\paragraph{Proof of Lemma \ref{pro-de}} The proof of \eqref{ds-F} can be found in \cite[Proposition 3.6]{DXZ22a}.

To show \eqref{drho-F}, we note from \eqref{def-fh} and \eqref{partial-rho-Psi} that
\begin{equation}\label{partia-rho-f}
\frac{\partial \mathbf{f}}{\partial \rho}(x) = \begin{pmatrix}
-2\rho/{3} & \alpha -1 & x\\
1 & \rho/3 & 0\\
0 & 1 & \rho/3
\end{pmatrix} \mathbf{f}(x), \qquad
\frac{\partial \mathbf{h}}{\partial \rho}(y) =
-\begin{pmatrix}
- 2 \rho/3 & 1 & 0 \\
\alpha -1 & \rho/3 & 1\\
y & 0 & \rho/3
\end{pmatrix} \mathbf{h}(y).
\end{equation}
This, together with \eqref{kernel2}, implies that
\begin{equation}\label{partial-rho-kernel}
\frac{\partial}{\partial \rho} K_{\alpha} (x, y; \rho) = \frac{\frac{\partial \mathbf{f}^{\msf T}}{\partial \rho}(x)\mathbf{h}(y)+\mathbf{f}(x)^{\msf T}\frac{\partial \mathbf{h}}{\partial \rho}(y)}{x-y} = \mathbf{f}(x)^{\msf T}
\begin{pmatrix}
0 & 0 & 0\\
0 & 0 & 0\\
1 & 0 & 0
\end{pmatrix}\mathbf{h}(y)=f_3(x)h_1(y).
\end{equation}
Thus, it is readily seen from \eqref{def-F} and \eqref{def-FH} that
\begin{equation}\label{partial-rho-F}
\frac{\partial}{\partial \rho} F(s; \rho) = \frac{\partial}{\partial \rho} \ln{\det{(I - \mathcal K_{s})}}
 = -\tr{\left((I - \mathcal K_{s})^{-1} \frac{\partial}{\partial \rho}K_{\alpha}\right)} = -\int_0^s F_3(v) h_1(v) \ud v.
\end{equation}
On the other hand, it follows from \eqref{infty-Y} and \eqref{def-Y} that
\begin{equation}
Y_1 = \int_0^s \mathbf{F}(w) \mathbf{h}(w)^{\msf T} \ud w = \int_0^s \begin{pmatrix}
F_1(w)\\
F_2(w)\\
F_3(w)
\end{pmatrix} \begin{pmatrix}
h_1(w) & h_2(w) & h_3(w)
\end{pmatrix} \ud w.
\end{equation}
The above two formulas gives us
\begin{equation}\label{eq:FY1}
\frac{\partial}{\partial \rho} F(s; \rho) = -(Y_1)_{31}.
\end{equation}
We finally arrive at the differential identity \eqref{drho-F} by combining \eqref{eq:FY1}, \eqref{def-Psi-0}, \eqref{def-X1} and a straightforward calculation.

This completes the proof of Lemma \ref{pro-de}.
\qed

\section{Meromorphic $\lambda$-functions on a Riemann surface} \label{sec:lambdafunctions}
It is the aim of this section to introduce the so-called $\lambda$-functions and to investigate their properties. These auxiliary functions
will be used to `partially' normalize the large-$z$ asymptotics of the scaled RH problem \ref{rhp-X} for $X$. In particular, the analytic continuation of $\lambda$-functions defines a meromorphic on a specific Riemann surface. This Riemann surface consists of three sheets $\mathcal{R}_j$, $j=1,2,3$, given by
\begin{equation*}
    \mathcal{R}_1 = \mathbb{C}\setminus [1, +\infty), \qquad \mathcal{R}_2 = \mathbb{C}\setminus ((-\infty, 0] \cup [1, +\infty)), \qquad \mathcal{R}_3 = \mathbb{C}\setminus (-\infty, 0].
\end{equation*}
The sheet $\mathcal R_1$ is connected to the sheet $\mathcal R_2$ through $[1,+\infty)$ and $\mathcal R_2$ is connected to $\mathcal R_3$ through $(-\infty,0]$. All these gluings are performed in the usual crosswise manner; see Figure \ref{fig:RS}. By adding a common point at $\infty$ to the three sheets, we obtain a compact Riemann surface of genus zero denoted by $\mathcal R$.

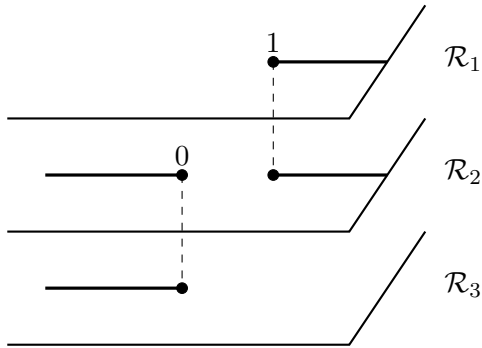
\begin{figure}[h]
\begin{center}
\begin{tikzpicture}

%
%Sheets
%
\draw[thick] (-0.5,4.5)--(4,4.5)--(5,6);
\draw[thick] (-0.5,3)--(4,3)--(5,4.5);
\draw[thick] (-0.5,1.5)--(4,1.5)--(5,3);
%\draw[thick] (-0.5,0)--(4,0)--(5,1.5);

\node[above] at (5.5,5) {$\mathcal R_1$};
\node[above] at (5.5,3.5) {$\mathcal R_2$};
\node[above] at (5.5,2) {$\mathcal R_3$};
%\node[above] at (5.5,.5) {$\mathcal R_4$};

%
%Branch cuts
%
%\draw[line width=1.2pt] (0,5.25)--(2,5.25);
\draw[line width=1.2pt] (0,3.75)--(1.8,3.75);
\draw[line width=1.2pt] (4.5,3.75)--(3,3.75);
\draw[line width=1.2pt] (4.5,5.25)--(3,5.25);
\draw[line width=1.2pt] (0,2.25)--(1.8,2.25);
%\draw[line width=1.2pt] (3,2.25)--(0,2.25);
%\draw[line width=1.2pt] (3,0.75)--(0,0.75);
%
%Branch points
%
\filldraw [black] (3,5.25) circle (2pt) node [above] (q1) {$1$};
\filldraw [black] (1.8,3.75) circle (2pt) node [above] (q2) {$0$};
\filldraw [black] (3,3.75) circle (2pt) node [above] (02) {};
%\filldraw [black] (4,3.75) circle (2pt) node [above] (p2) {$p$};
\filldraw [black] (1.8,2.25) circle (2pt) node [above] (03) {};
%\filldraw [black] (4,2.25) circle (2pt) node [above] (p3) {};
%\filldraw [black] (3,0.75) circle (2pt) node [above] (04) {};

%
%Conections
%
\draw[dashed] (3,5.25)--(3,3.75);
\draw[dashed] (1.8,3.75)--(1.8,2.25);
%\draw[dashed] (3,3.75)--(3,2.25);

\end{tikzpicture}
\end{center}
\caption{The Riemann surface $\mathcal R$.}
\label{fig:RS}
\end{figure}

For each $j=1,2,3$, we will construct a function $\lambda_j$, which is analytic on $\mathcal{R}_j$ and admits an analytic continuation across the cuts. The construction, however, is indirect in the sense that the $\lambda$-functions are built in terms of the $w$-functions introduced next.

\subsection{The $w$-functions}
%\begin{figure}[ht]
%  \centering
%  \includegraphics[scale=.6]{Reimann surface.pdf}
%  \caption{The sheets of Riemann surface $\mathcal{R}$.}
%  \label{figure-RS}
%\end{figure}
%We introduce a three-sheeted Riemann surface $\mathcal{R}$ with sheets
%
%and glue them together along the cuts $(-\infty, 0]$ and $[1, +\infty)$ in the usual crosswise manner. We denote this compact Riemann surface by $\mathcal{R}$, which has genus zero and is shown in Figure \ref{figure-RS}.
%
%We intend to find functions $\lambda_j$, $j=1, 2, 3$ on these sheets, such that each $\lambda_j$ is analytic on $\mathcal{R}_j$ and admits an analytic continuation across the cuts. For this purpose, we start with an elementary function $w(z)$ that is meromorphic on $\mathcal{R}$,
The $w$-functions are three solutions of the algebraic equation
\begin{equation}\label{algebraic equation}
    w(z)^3-\frac{3}{2} w(z)^2+\frac{z}{2}=0.
\end{equation}
It is straightforward to check that the discriminant of \eqref{algebraic equation} is $-27(z-1)z/4$. Its two roots along with the point at infinity constitute the three branch points of the Riemann surface $\mathcal{R}$. By using Cardano's formula, the three solutions of \eqref{algebraic equation} are explicitly given by
\begin{align}
   w_1(z) &=\frac{1}{2}\left(\eta (z)^{\frac{1}{3}}+ \eta (z)^{-\frac{1}{3}}+1\right),\label{def-w1}\\
   w_2(z) &=\frac{1}{2}\left(\omega^{-1} \eta (z)^{\frac{1}{3}}+\omega \eta (z)^{-\frac{1}{3}}+1\right),\label{def-w2}\\
   w_3(z) &=\frac{1}{2}\left(\omega \eta (z)^{\frac{1}{3}}+\omega^{-1} \eta (z)^{-\frac{1}{3}}+1\right),\label{def-w3}
   \end{align}
    where $\omega=e^{2 \pi \ii /3}$,
    \begin{equation}\label{def-eta}
        \eta (z) = 2\sqrt{z(z-1)}+1-2z, \quad z \in \mathbb{C} \setminus ((-\infty, 0] \cup [1, +\infty)),
    \end{equation}
    with
    \begin{equation}\label{arg-eta}
        \arg \eta(z) \in (0, \pi).
    \end{equation}
Indeed, it is readily seen that $\eta (z)$ satisfies the quadratic equation
\begin{equation}\label{eq:etaequ}
\eta (z)^2 + 4z \eta (z) - 2\eta (z) +1=0.
\end{equation}
Thus, for $j=1,2,3$, we obtain from \eqref{def-w1}--\eqref{def-w3} that
\begin{equation}
w_j(z)^3-\frac{3}{2} w_j(z)^2+\frac{z}{2}= \frac{1}{8} \left(\eta (z) + \eta (z)^{-1} + 4z -2\right) = 0,
\end{equation}
where we have made use of \eqref{eq:etaequ} and the fact that $\eta(z)\neq 0$ in the last step. The condition \eqref{arg-eta} follows from the observation that $\Im \eta (z)>0$ for $z \in \mathbb{C} \setminus ((-\infty, 0] \cup [1, +\infty))$.
%We choose three solutions $w_j(z)$, $j=1, 2, 3$ to \eqref{algebraic equation} such that they are defined and analytic on $\mathcal{R}_j$, respectively.
%
%\begin{proposition}
%    The three solutions $w_j(z)$, $j=1, 2, 3$ to the algebraic equation \eqref{algebraic equation} are given by
%
%\end{proposition}
%
%\begin{proof}
%It is clear that $\eta (z) \neq 0$ and satisfies
%\begin{equation}
%\eta (z)^2 + 4z \eta (z) - 2\eta (z) +1=0.
%\end{equation}
%Hence, with $w_j(z)$, $j = 1, 2, 3,$ defined in \eqref{def-w1}, \eqref{def-w2} and \eqref{def-w3}, we have
%\begin{equation}
%w_j(z)^3-\frac{3}{2} w(z)_j^2+\frac{z}{2}= \frac{1}{8} \left(\eta (z) + \eta (z)^{-1} + 4z -2\right) = 0.
%\end{equation}
%This finishes our proof.
%\end{proof}

Some properties of the $w$-functions are collected in the following proposition.
\begin{proposition}\label{prop:wfunction}
The functions $w_j (z)$, $j = 1, 2, 3$, given in \eqref{def-w1}--\eqref{def-w3} satisfy the following properties.
\begin{enumerate}
\item [\rm (i)] $w_j(z)$ is analytic on the sheet $\mathcal{R}_j$ and satisfies
\begin{align}
w_{2, \pm}(x) &= w_{3, \mp}(x), &&  x \in (-\infty, 0), \label{w2-w3}\\
w_{2, \pm}(x) &= w_{1, \mp}(x), &&  x \in (1, \infty). \label{w2-w1}
\end{align}
Here, we orient $(-\infty, 0)$ and $(1, \infty)$ from the left to the right. Hence, the function
\begin{equation}\label{def:w}
w: \cup_{j=1}^3 \mathcal {R}_j \to \overline{\mathbb{C}},\qquad w|_{\mathcal {R}_j}=w_j,
\end{equation}
extends to a meromorphic function on the Riemann surface $\mathcal R$. This function is a bijection.

%$\cup_{j=1}^3 \mathcal{R}_j \to \mathbb{C}: \mathcal{R}_j \ni z \mapsto w_j(z)$ has an analytic continuation to a meromorphic function $w: \mathcal{R} \to \bar{\mathbb{C}}$. This function is a bijection.
\item [\rm (ii)] As $z \to \infty$ with $-\pi < \arg{z} < \pi$, we have
\begin{equation}\label{infty-w2}
w_2(z) = \begin{cases}
-2^{-\frac{1}{3}} \omega^2 z^{\frac{1}{3}} + \frac{1}{2} - \frac{\omega}{2^{5/3}} z^{-\frac{1}{3}} + \frac{\omega^2}{6 \cdot 2^{1/3}} z^{-\frac{2}{3}} - \frac{\omega}{6 \cdot 2^{5/3}} z^{-\frac{4}{3}} + \Boh (z^{-\frac{5}{3}}), & \quad \Im{z} > 0,\\
-2^{-\frac{1}{3}} \omega z^{\frac{1}{3}} + \frac{1}{2} - \frac{\omega^2}{2^{5/3}} z^{-\frac{1}{3}} + \frac{\omega}{6 \cdot 2^{1/3}} z^{-\frac{2}{3}} - \frac{\omega^2}{6 \cdot 2^{5/3}} z^{-\frac{4}{3}} + \Boh (z^{-\frac{5}{3}}), & \quad \Im{z} < 0,
\end{cases}
\end{equation}
and
\begin{multline}\label{infty-w3}
w_3(z) = -2^{-\frac{1}{3}} z^{\frac{1}{3}} + \frac{1}{2} - \frac{1}{2^{5/3}} z^{-\frac{1}{3}} + \frac{1}{6 \cdot 2^{1/3}} z^{-\frac{2}{3}}
\\
- \frac{1}{6 \cdot 2^{5/3}} z^{-\frac{4}{3}} + \Boh ( z^{-\frac{5}{3}}), \qquad z \in \mathbb{C} \setminus (-\infty, 0].
\end{multline}
\item [\rm (iii)] As $z \to 0$ with $-\pi < \arg{z} < \pi$, we have
\begin{equation}\label{0-w2}
w_2(z) = \frac{\sqrt{3}}{3} z^{\frac{1}{2}} + \frac{z}{9} + \frac{5 \sqrt{3}}{162} z^{\frac{3}{2}} + \Boh ( z^{\frac{5}{2}}),
\end{equation}
and
\begin{equation}\label{0-w3}
w_3(z) = -\frac{\sqrt{3}}{3} z^{\frac{1}{2}} + \frac{z}{9} - \frac{5 \sqrt{3}}{162} z^{\frac{3}{2}} + \Boh ( z^{\frac{5}{2}}).
\end{equation}
\item [\rm (iv)] As $z \to 1$ with $- \pi < \arg{(z-1)} < \pi$, we have
\begin{equation}\label{1-w1}
w_1(z) = \begin{cases}
1 - \frac{\ii}{\sqrt{3}} (z-1)^{\frac{1}{2}} + \frac{1}{9}(z-1) +\frac{5 \sqrt{3}\ii}{162}(z-1)^{\frac{3}{2}}
 \\
 ~ - \frac{8}{243}(z-1)^2+ \Boh ((z-1)^{\frac{5}{2}}), & \Im{z} > 0,\\
1 + \frac{\ii}{\sqrt{3}} (z-1)^{\frac{1}{2}} + \frac{1}{9}(z-1) - \frac{5 \sqrt{3}\ii}{162}(z-1)^{\frac{3}{2}}
\\
~ - \frac{8}{243}(z-1)^2+ \Boh ((z-1)^{\frac{5}{2}}), & \Im{z} < 0,
\end{cases}
\end{equation}
and
\begin{equation}\label{1-w2}
w_2(z) = \begin{cases}
1 + \frac{\ii}{\sqrt{3}} (z-1)^{\frac{1}{2}} + \frac{1}{9}(z-1) -\frac{5 \sqrt{3}\ii}{162}(z-1)^{\frac{3}{2}}
 \\
 ~ - \frac{8}{243}(z-1)^2+ \Boh ((z-1)^{\frac{5}{2}}), & \Im{z} > 0,
 \\
1 - \frac{\ii}{\sqrt{3}} (z-1)^{\frac{1}{2}} + \frac{1}{9}(z-1) + \frac{5 \sqrt{3}\ii}{162}(z-1)^{\frac{3}{2}}
 \\
 ~ - \frac{8}{243}(z-1)^2+ \Boh ((z-1)^{\frac{5}{2}}), & \Im{z} < 0.
\end{cases}
\end{equation}
\end{enumerate}
\end{proposition}

\begin{proof}
To prove \eqref{w2-w3}, we see from the definition of $\eta (z)$ in \eqref{def-eta} that for $x<0$,
\begin{equation}
\eta_{\pm} (x) = 1 - 2x \mp \sqrt{x(x-1)} \quad  \textrm{and} \quad \eta_+(x) \eta_-(x) = 1
\end{equation}
Thus, from the definition of $w_2(z)$ in \eqref{def-w2}, it follows that
\begin{equation*}
w_{2, +}(x) = \frac{1}{2} \left(\omega^{-1} \eta_+(x)^{\frac{1}{3}} + \omega \eta_+(x)^{-\frac{1}{3}}+1\right) = \frac{1}{2} \left(\omega^{-1} \eta_-(x)^{-\frac{1}{3}} + \omega \eta_-(x)^{\frac{1}{3}}+1 \right) = w_{3, -}(x).
\end{equation*}
Similarly, we can obtain $w_{2, -}(x) = w_{3, +}(x)$ for $x<0$ and \eqref{w2-w1}.

Next, we come to the asymptotics of $w_j(z), j = 2, 3$, as $z \to \infty$. From \eqref{def-eta}, it is easily seen that as $z\to \infty$,
\begin{equation}\label{infty-eta}
\eta(z) = \begin{cases}
-\frac{1}{4z} - \frac{1}{8z^2} - \frac{5}{64 z^3} +  \Boh (z^{-4}), & \quad \Im{z} > 0,\\
-4z + 2 +\frac{1}{4z} + \frac{1}{8z^2} + \frac{5}{64 z^3} +  \Boh (z^{-4}), & \quad \Im{z} < 0.
\end{cases}
\end{equation}
Substituting the above formula into \eqref{def-w2}, it is readily seen that, as $z \to \infty$,
\begin{align*}
&w_2(z) = \frac{1}{2} \left(\omega^{-1} \eta (z)^{\frac{1}{3}}+\omega \eta (z)^{-\frac{1}{3}}+1\right)\\
& = \frac{1}{2} \left(\frac{e^{- \pi \ii/3}}{2^{2/3}z^{1/3}}\left(1+\frac{1}{2z} +\frac{5}{16 z^2} +  \Boh (z^{-3})\right)^{\frac{1}{3}}+e^{\frac{\pi \ii}{3}} 2^{\frac{2}{3}} z^{\frac{1}{3}} \left(1+\frac{1}{2z} +\frac{5}{16 z^2} +  \Boh (z^{-3})\right)^{-\frac{1}{3}}+1\right)\\
& = -2^{-\frac{1}{3}} \omega^2 z^{\frac{1}{3}} + \frac{1}{2} - \frac{\omega}{2^{5/3}} z^{-\frac{1}{3}} + \frac{\omega^2}{6 \cdot 2^{1/3}} z^{-\frac{2}{3}} - \frac{\omega}{6 \cdot 2^{5/3}} z^{-\frac{4}{3}} + \Boh (z^{-\frac{5}{3}}),\quad \Im{z} > 0,
\end{align*}
and
\begin{equation*}
\begin{split}
&w_2(z) = \frac{1}{2} \left(\omega^{-1} \eta (z)^{\frac{1}{3}}+\omega \eta (z)^{-\frac{1}{3}}+1\right)\\
& = \frac{1}{2} \left(e^{-\frac{\pi \ii}{3}} 2^{\frac{2}{3}} z^{\frac{1}{3}}\left(1-\frac{1}{2z} -\frac{1}{16 z^2} +  \Boh (z^{-3})\right)^{\frac{1}{3}}+\frac{e^{\pi \ii/3}}{2^{2/3}z^{1/3}} \left(1-\frac{1}{2z} -\frac{1}{16 z^2} +  \Boh (z^{-3})\right)^{-\frac{1}{3}}+1\right)\\
& = -2^{-\frac{1}{3}} \omega z^{\frac{1}{3}} + \frac{1}{2} - \frac{\omega^2}{2^{5/3}} z^{-\frac{1}{3}} + \frac{\omega}{6 \cdot 2^{1/3}} z^{-\frac{2}{3}} - \frac{\omega^2}{6 \cdot 2^{5/3}} z^{-\frac{4}{3}} + \Boh (z^{-\frac{5}{3}}), \quad \Im{z} < 0,
\end{split}
\end{equation*}
which is \eqref{infty-w2}. The asymptotics of $w_3(z)$ in \eqref{infty-w3} can be obtained through the same fashion.

We then show the asymptotics of $w_j(z)$, $j= 2, 3$, as $z \to 0$. It follows from \eqref{def-eta} that, as $z \to 0$,
\begin{equation}
\eta (z) = 1 + 2 \ii z^{\frac{1}{2}} - 2 z - \ii z^{\frac{3}{2}} - \frac{\ii}{4} z^{\frac{5}{2}} + \Boh (z^{\frac{7}{2}}).
\end{equation}
Inserting the above formula into \eqref{def-w2} and \eqref{def-w3} gives us
\begin{equation*}
\begin{split}
&w_2(z) = \frac{1}{2} \left(\omega^{-1} \eta (z)^{\frac{1}{3}}+\omega \eta (z)^{-\frac{1}{3}}+1\right)\\
& = \frac{1}{2} \left(e^{-\frac{2 \pi \ii}{3}} \left(1 + \frac{2 \ii}{3}z^{\frac{1}{2}} - \frac{2 z}{9} +\frac{5 \ii}{81} z^{\frac{3}{2}}+ \Boh (z^{\frac{5}{2}})\right)+e^{\frac{2 \pi \ii}{3}} \left(1 - \frac{2 \ii}{3}z^{\frac{1}{2}} - \frac{2 z}{9} -\frac{5 \ii}{81} z^{\frac{3}{2}}+ \Boh (z^{\frac{5}{2}})\right)+1\right)\\
& =\frac{\sqrt{3}}{3} z^{\frac{1}{2}} + \frac{z}{9} + \frac{5 \sqrt{3}}{162} z^{\frac{3}{2}} + \Boh ( z^{\frac{5}{2}}),
\end{split}
\end{equation*}
and
\begin{equation*}
\begin{split}
&w_3(z) = \frac{1}{2} \left(\omega \eta (z)^{\frac{1}{3}}+\omega^{-1} \eta (z)^{-\frac{1}{3}}+1\right)\\
& = \frac{1}{2} \left(e^{\frac{2 \pi \ii}{3}} \left(1 + \frac{2 \ii}{3}z^{\frac{1}{2}} - \frac{2 z}{9} +\frac{5 \ii}{81} z^{\frac{3}{2}}+ \Boh (z^{\frac{5}{2}})\right)+e^{-\frac{2 \pi \ii}{3}} \left(1 - \frac{2 \ii}{3}z^{\frac{1}{2}} - \frac{2 z}{9} -\frac{5 \ii}{81} z^{\frac{3}{2}}+ \Boh (z^{\frac{5}{2}})\right)+1\right)\\
& =-\frac{\sqrt{3}}{3} z^{\frac{1}{2}} + \frac{z}{9} -\frac{5 \sqrt{3}}{162} z^{\frac{3}{2}} + \Boh ( z^{\frac{5}{2}}),
\end{split}
\end{equation*}
which is \eqref{0-w2} and \eqref{0-w3}.

Finally, we move to the asymptotics $w_j(z)$, $j =1, 2$, as $z \to 1$. We note that, as $z \to 1$,
\begin{equation}
\eta(z)=
\begin{cases}
-1 + 2 (z-1)^{\frac{1}{2}} - 2(z-1) + (z-1)^{\frac{3}{2}} +  \Boh ((z-1)^{\frac{5}{2}}), & \quad \Im{z}>0,\\
-1 - 2 (z-1)^{\frac{1}{2}} - 2(z-1) - (z-1)^{\frac{3}{2}} +  \Boh ((z-1)^{\frac{5}{2}}), & \quad \Im{z}>0.
\end{cases}
\end{equation}
Substituting the above formula into \eqref{def-w1} and \eqref{def-w2} gives us \eqref{1-w1} and \eqref{1-w2} after direct calculations.

This finishes the proof of Proposition \ref{prop:wfunction}.
\end{proof}

It is easily seen from the above proposition that the branch points of $\mathcal{R}$--$0,1,\infty$, are mapped to the points $0,1,\infty$, on the $w$-sphere. Bijection \eqref{def:w} between the Riemann surface $\mathcal{R}$ and the extended $w$-plane are illustrated in Figure \ref{figure-w}.

\begin{figure}[ht]
 \centering
  \includegraphics[scale=.6]{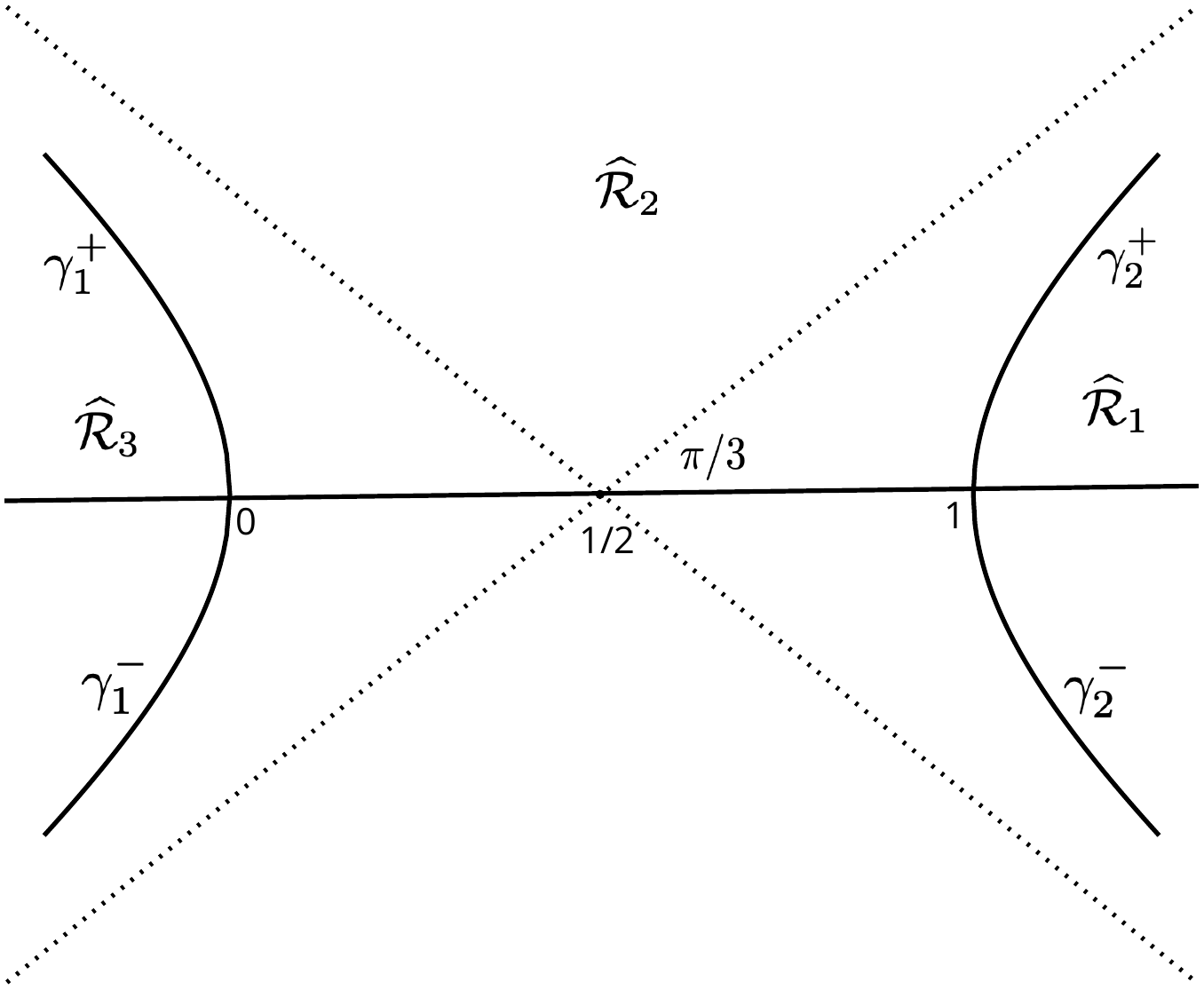}
  \caption{Image of the map $w$: $\mathcal{R} \mapsto \overline{\mathbb{C}}$. The solid lines $\gamma_i^{\pm}$, $i = 1, 2$, are the images of the cuts in the Riemann surface $\mathcal{R}$ under this map. More precisely, $\gamma_1^{\pm} = w_{2, \pm} (-\infty,0)$, $\gamma_2^{\pm} =  w_{2, \pm} (1, \infty)$ and $w(\mathcal{R}_k) = \widehat{\mathcal{R}}_k$, $k = 1, 2, 3$.}
  \label{figure-w}
\end{figure}

\subsection{The $\lambda$-functions}
With the $w$-functions given in \eqref{def-w1}--\eqref{def-w3}, the $\lambda$-functions are defined by
\begin{equation}\label{def-lambda}
	\lambda_j(z) = \frac{3}{2^{1/3}}w_j(z)^2 - \left(\frac{3}{2^{1/3}}+\frac{ 2^{1/3}}{s^{1/3}}\rho \right)w_j(z) - \frac{3}{2^{7/3}} + \frac{\rho}{2^{2/3} s^{1/3}}, \qquad j=1,2,3,
\end{equation}
which depend on the parameters $s > 0$ and $\rho \in \mathbb{R}$. In view of Proposition \ref{prop:wfunction}, the following properties of the $\lambda$-functions follow directly from \eqref{def-lambda} and straightforward calculations.

\begin{proposition}\label{pro-lambda}
The functions $\lambda_j(z)$, $j = 1,2,3$, defined in \eqref{def-lambda} have the following properties.
\begin{enumerate}
\item [\rm (i)] $\lambda_j(z)$ is analytic on $\mathcal{R}_j$ and satisfies
\begin{align}
\lambda_{2, \pm}(x) &= \lambda_{3, \mp}(x), &&  x \in (-\infty, 0), \\
\lambda_{2, \pm}(x) &= \lambda_{1, \mp}(x), &&  x \in (1, \infty).
\end{align}
Hence, the function
\begin{equation}\label{def:lambda}
\lambda: \cup_{j=1}^3 \mathcal {R}_j \to \overline{\mathbb{C}},\qquad \lambda|_{\mathcal {R}_j}=\lambda_j,
\end{equation}
extends to a meromorphic function on the Riemann surface $\mathcal R$.
%$\cup_{j=1}^3 \mathcal{R}_j \to \mathbb{C}: \mathcal{R}_j \ni z \mapsto \lambda_j(z)$ has an analytic continuation to a meromorphic function $w: \mathcal{R} \to \bar{\mathbb{C}}$. This function is a bijection.
\item [\rm (ii)] As $z \to \infty$ with $-\pi < \arg{z} < \pi$, we have
\begin{equation}\label{infty-lambda2}
\lambda_2 (z) = \begin{cases}
\frac{3}{2}\omega z^{\frac{2}{3}} + \frac{\rho \omega^2}{s^{1/3}} z^{\frac{1}{3}} + \omega d_1z^{-\frac{1}{3}} + \omega^2 d_2 z^{-\frac{2}{3}} + \Boh(z^{-1}), & \quad \Im{z} > 0,\\
\frac{3}{2}\omega^2 z^{\frac{2}{3}} + \frac{\rho \omega}{s^{1/3}} z^{\frac{1}{3}} + \omega^2 d_1z^{-\frac{1}{3}} + \omega d_2 z^{-\frac{2}{3}} + \Boh(z^{-1}), & \quad \Im{z} < 0,
\end{cases}
\end{equation}
and
\begin{equation}\label{infty-lambda3}
\lambda_3 (z) = \frac{3 }{2} z^{\frac{2}{3}} + \frac{\rho}{s^{1/3} }z^{\frac{1}{3}} + d_1 z^{-\frac{1}{3}} + d_2 z^{-\frac{2}{3}} + \Boh(z^{-1}), \quad z \in \mathbb{C} \setminus (-\infty, 0],
\end{equation}
where
\begin{equation}\label{d12}
d_1 = - \frac{1}{2} + \frac{\rho}{2^{4/3}s^{1/3}},\qquad d_2 = \frac{3}{2^{11/3}} - \frac{\rho}{6s^{1/3}}.
\end{equation}
\item [\rm (iii)] As $z \to 0$ with $-\pi < \arg{z} < \pi$, we have
\begin{equation}\label{0-lambda2}
\lambda_2 (z) = c_0 + c_1 z^{\frac{1}{2}} + c_2 z + c_3 z^{\frac{3}{2}} + \Boh(z^2),
\end{equation}
and
\begin{equation}\label{0-lambda3}
\lambda_3 (z) = c_0 - c_1 z^{\frac{1}{2}} + c_2 z - c_3 z^{\frac{3}{2}} + \Boh(z^2),
\end{equation}
where
\begin{equation}\label{tildec}
\begin{aligned}
c_0 &=-\frac{3}{2^{7/3}}+\frac{\rho}{2^{2/3} s^{1/3}},   && \quad  c_1 = - \frac{\sqrt{3}}{2^{1/3}} - \frac{2^{1/3} \rho}{\sqrt{3} s^{1/3}},\\
c_2 &= \frac{2^{2/3}}{3} - \frac{2^{1/3} \rho}{9 s^{1/3}},  && \quad c_3 = \frac{7 \cdot 2^{2/3}}{36 \sqrt{3}}- \frac{5\cdot 2^{1/3} \rho}{54 \sqrt{3}s^{1/3}}.
\end{aligned}
\end{equation}
\item [\rm (iv)] As $z \to 1$ with $- \pi < \arg{(z-1)} < \pi$, we have
\begin{equation}\label{1-lambda1}
\lambda_1(z) =\begin{cases}
\tilde{c}_0 - \ii \tilde{c}_1 (z-1)^{\frac{1}{2}} + \tilde{c}_2 (z-1) - \ii \tilde{c}_3 (z-1)^{\frac{3}{2}} + \Boh((z-1)^2), & \quad \Im{z}> 0,\\
\tilde{c}_0 + \ii \tilde{c}_1 (z-1)^{\frac{1}{2}} + \tilde{c}_2 (z-1) + \ii \tilde{c}_3 (z-1)^{\frac{3}{2}} + \Boh((z-1)^2), & \quad \Im{z}< 0,
\end{cases}
\end{equation}
and
\begin{equation}\label{1-lambda2}
\lambda_2(z) =\begin{cases}
\tilde{c}_0 + \ii \tilde{c}_1 (z-1)^{\frac{1}{2}} + \tilde{c}_2 (z-1) + \ii \tilde{c}_3 (z-1)^{\frac{3}{2}} + \Boh((z-1)^2), & \quad \Im{z}> 0,\\
\tilde{c}_0 - \ii \tilde{c}_1 (z-1)^{\frac{1}{2}} + \tilde{c}_2 (z-1) - \ii \tilde{c}_3 (z-1)^{\frac{3}{2}} + \Boh((z-1)^2), & \quad \Im{z}< 0,
\end{cases}
\end{equation}
where
\begin{equation}\label{def-c}
\begin{aligned}
\tilde{c}_0 & = -\frac{3}{2^{7/3}} -\frac{\rho}{2^{2/3} s^{1/3}}, && \quad  \tilde{c}_1 = \frac{\sqrt{3}}{2^{1/3}} -\frac{2^{1/3} \rho}{\sqrt{3} s^{1/3}},
\\
\tilde{c}_2 & = \frac{2^{2/3}}{3} + \frac{2^{1/3} \rho}{9 s^{1/3}}, &&\quad \tilde{c}_3 = \frac{7 \cdot 2^{2/3}}{36 \sqrt{3}} +  \frac{5\cdot 2^{1/3} \rho}{54 \sqrt{3}s^{1/3}}.
\end{aligned}
\end{equation}
\end{enumerate}
\end{proposition}

In view of items (i) and (ii) in Proposition \ref{pro-lambda} and \eqref{def-theta-k}, it is readily seen that, as $z\to \infty$,
\begin{equation}\label{eq:lambda1theta}
\lambda_1(z)=\left\{
               \begin{array}{ll}
                 s^{-\frac23}\theta_1(sz)+\Boh(z^{-\frac13}), & \hbox{$\Im z > 0$,} \\
                 s^{-\frac23}\theta_2(sz)+\Boh(z^{-\frac13}), & \hbox{$\Im z < 0$,}
               \end{array}
             \right.
\end{equation}
\begin{equation}
\lambda_2(z)=\left\{
               \begin{array}{ll}
                 s^{-\frac23}\theta_2(sz)+\Boh(z^{-\frac13}), & \hbox{$\Im z > 0$,} \\
                 s^{-\frac23}\theta_1(sz)+\Boh(z^{-\frac13}), & \hbox{$\Im z < 0$,}
               \end{array}
             \right.
\end{equation}
and
\begin{equation}\label{eq:lambda3theta}
\lambda_3(z)=s^{-\frac23}\theta_3(sz)+\Boh(z^{-\frac13}).
\end{equation}

\section{Asymptotic analysis of the RH problem for $X$}\label{sec:asyanaly}
In this section, we will  perform a Deift-Zhou steepest descent analysis \cite{Deift1999} for the RH problem for $X$. It consists of a series of explicit and invertible transformations and the final goal is to arrive at an RH problem tending to the identity matrix as $s\to \infty$.

\subsection{First transformation: $X \to T$}

The first transformation is a rescaling of the RH problem for $X$, which is defined by
\begin{equation}\label{def-T}
    T(z) = X(sz).
\end{equation}
It is then easily seen from the RH problem \ref{rhp-X} for $X$ that $T(z)$ satisfies the following RH problem.
\begin{rhp}\label{RHP:T}
\hfill
    %The function $T(z)$ defined in \eqref{def-T} has the following properties:
    \begin{itemize}
            \item [\rm (a)] $T (z)$ is analytic in $\mathbb{C} \setminus \Sigma_{T}$, where
           \begin{equation}\label{def:sigmaT}
    \Sigma_T:=\cup_{i=2,3,4} \Sigma_{i} \cup \{0\} \cup_{i=0,1,5} \Sigma_{i}^{(1)}\cup\{1\},
    \end{equation}
    see the solid lines in Figure \ref{figure-X} with $s=1$.
    \item [\rm (b)] For $z \in \Sigma_{T}\setminus\{0,1\}$, we have
    \begin{equation}
        T_+(z) = T_-(z) J_{T}(z),
    \end{equation}
    where
    \begin{equation}\label{jump-T}
        J_{T}(z) = \begin{cases}
        J_{\Psi}(z), &\quad z \in \cup_{i=2}^4\Sigma_i,\\
        \begin{pmatrix}
	    0 & 1 & 0\\
	    -1 & 0 & 0\\
	    0 & 0 & 1
	\end{pmatrix}, &\quad z \in \Sigma_0^{(1)},\\
	\begin{pmatrix}
	     1 & 0 & 0\\
	     1 & 1 & 0\\
	     0 & 0 & 1
	\end{pmatrix}, &\quad z \in \Sigma_1^{(1)},\\
	\begin{pmatrix}
	     1 & 0 & 0\\
	     1 & 1 & 0\\
	     0 & 0 & 1
	\end{pmatrix}, &\quad z \in \Sigma_5^{(1)},
        \end{cases}
    \end{equation}
    with $J_\Psi$ given in \eqref{jump-Psi}.

    \item [\rm (c)] As $z \to \infty$ with $z\in \mathbb{C}\setminus \Sigma_T$, we have
     \begin{multline}\label{infty-T}
        T(z)  = \frac{\ii z^{-\alpha/3}}{\sqrt{3}}\Psi_{0} \left(I + \frac{X_1}{sz} + \Boh(z^{-2})\right)
        \diag{((sz)^{\frac{1}{3}}, 1, (sz)^{-\frac{1}{3}})}
        \\
        \times L_{\pm} \diag{(e^{\pm \frac{\alpha \pi }{3}\ii}, e^{\mp \frac{\alpha \pi }{3}\ii}, 1)} e^{\Theta (sz)}, \quad \pm \Im{z} > 0,
    \end{multline}
where $\Psi_{0}$, $X_1$, $L_{\pm}$ and $\Theta (z)$ are given in \eqref{def-Psi-0}, \eqref{def-X1}, \eqref{def-L+-} and \eqref{def-theta}, respectively.
    %\begin{equation}\label{infty-T}
%        T (z) = \frac{\ii (sz)^{-\frac{\alpha}{3}}}{\sqrt{3}} \Psi_{0} \left(I + \frac{X_1}{sz} + \Boh((sz)^{-2})\right) \diag{((sz)^{\frac{1}{3}}, 1, (sz)^{-\frac{1}{3}})}L_{\pm} \diag{(e^{\pm \frac{\alpha \pi \ii}{3}}, e^{\mp \frac{\alpha \pi \ii}{3}}, 1)} e^{\Theta (sz)},
%    \end{equation}
%    where $L_{\pm}$ and $\Theta (z)$ are given in \eqref{def-L+-} and \eqref{def-theta}.
    \item [\rm (d)] As $z \to 0$, we have
    \begin{equation}\label{0-T}
        T (z) = \begin{cases}
            \Boh(z^{-\alpha}), & \quad \alpha>0,\\
            \Boh(\ln{z}), & \quad \alpha=0,\\
            \Boh(1), & \quad -1 < \alpha <0.
        \end{cases}
    \end{equation}
    \item [\rm (e)] As $z \to 1$, we have
    \begin{equation}\label{1-T}
        T (z) = \Boh(\ln{(z-1)}).
    \end{equation}
    \end{itemize}
\end{rhp}

\subsection{Second transformation: $T \to S$}
On account of \eqref{eq:lambda1theta}--\eqref{eq:lambda3theta}, we use the $\lambda$-functions to `partially' normalize the large-$z$ asymptotics of $T$ in the second transformation. It is defined by
\begin{equation}\label{def-S}
	S(z) = - \ii \sqrt{3} s^{\frac{\alpha}{3}}S_0 \diag{(s^{-\frac{1}{3}}, 1, s^{\frac{1}{3}})}\Psi_{0}^{-1} T(z) \diag{(e^{-s^{2/3}\lambda_{1}(z)},e^{-s^{2/3}\lambda_{2}(z)},e^{-s^{2/3}\lambda_{3}(z)}}),
\end{equation}
where
\begin{equation}\label{def-S-0}
S_0 = \begin{pmatrix}
1 & s^{\frac{2}{3}}d_1 & \frac{s^{4/3}}{2}d_1^2+s^{\frac{2}{3}}d_2\\
0 & 1 & s^{\frac{2}{3}}d_1\\
0 & 0 & 1
\end{pmatrix}
\end{equation}
with $d_1$ and $d_2$ being the constants given in \eqref{d12}, $\Psi_0$ and the functions $\lambda_{i}$, $i = 1, 2, 3$, are defined in \eqref{def-Psi-0} and \eqref{def-lambda}, respectively. With the aid of Proposition \ref{pro-lambda} and RH problem \ref{RHP:T} for $T$, it is straightforward to check that $S(z)$ defined in \eqref{def-S} satisfies the following RH problem.
\begin{rhp}
   \hfill
    \begin{itemize}
            \item [\rm (a)] $S (z)$ is analytic in $\mathbb{C} \setminus \Sigma_{T}$; where $\Sigma_T$ is defined in \eqref{def:sigmaT}.
    \item [\rm (b)] For $z \in \Sigma_{T}\setminus \{0,1\}$, we have
    \begin{equation}
        S_+(z) = S_-(z) J_{S}(z),
    \end{equation}
    where
    \begin{equation}\label{jump-S}
        J_{S}(z) = \begin{cases}
        \begin{pmatrix}
	    0 & 1 & 0\\
	    -1 & 0 & 0\\
	    0 & 0 & 1
	\end{pmatrix}, &\quad z \in \Sigma_0^{(1)},\\
	\begin{pmatrix}
	     1 & 0 & 0\\
	     e^{s^{2/3}(\lambda_{2}(z)-\lambda_{1}(z))} & 1 & 0\\
	     0 & 0 & 1
	\end{pmatrix}, &\quad z \in \Sigma_1^{(1)},\\
	\begin{pmatrix}
	     1 & 0 & 0\\
	     0 & 1 & e^{\alpha \pi \ii}e^{s^{2/3}(\lambda_{2}(z)-\lambda_{3}(z))}\\
	     0 & 0 & 1
	\end{pmatrix}, &\quad z \in \Sigma_2,\\
	\begin{pmatrix}
	     1 & 0 & 0\\
	     0 & 0 & -e^{-\alpha \pi \ii}\\
	     0 & e^{-\alpha \pi \ii} & 0
	\end{pmatrix}, &\quad z \in \Sigma_3,\\
	\begin{pmatrix}
	     1 & 0 & 0\\
	     0 & 1 & e^{-\alpha \pi \ii}e^{s^{2/3}(\lambda_{2}(z)-\lambda_{3}(z))}\\
	     0 & 0 & 1
	\end{pmatrix}, &\quad z \in \Sigma_4,\\
	\begin{pmatrix}
	     1 & 0 & 0\\
	     e^{s^{2/3}(\lambda_{2}(z)-\lambda_{1}(z))} & 1 & 0\\
	     0 & 0 & 1
	\end{pmatrix}, &\quad z \in \Sigma_5^{(1)}.
        \end{cases}
    \end{equation}
    \item[\rm (c)] As $z \to \infty$ with $z\in \mathbb{C}\setminus \Sigma_T$, we have
    \begin{multline}\label{infty-S}
        S (z) = z^{-\frac{\alpha}{3}} \left(I + \frac{S_1}{z}+\Boh(z^{-2})\right)\diag{(z^{\frac{1}{3}}, 1, z^{-\frac{1}{3}})}
        \\
        \times L_{\pm} \diag{(e^{\pm \frac{\alpha \pi \ii}{3}}, e^{\mp \frac{\alpha \pi \ii}{3}}, 1)}, \quad \pm \Im z>0,
    \end{multline}
    where
    \begin{equation}\label{def-S1}
	S_1 = \begin{pmatrix}
	\ast & \ast & \ast\\
	\ast & \ast & \ast\\
	-s^{\frac{2}{3}}d_1 + s^{-\frac{1}{3}} (X_1)_{31} & \ast & \ast
	\end{pmatrix}
    \end{equation}
  with $d_1$ and $X_1$ given in \eqref{d12} and \eqref{def-X1}, and $L_{\pm}$ are given in \eqref{def-L+-}.
    \item [\rm (d)] $S(z)$ has the same local behaviors as $T$ near $z = 0$ and $z = 1$; see \eqref{0-T} and \eqref{1-T}.
   %  we have
%    \begin{equation}\label{0-S}
%        S (z) = \begin{cases}
%            \Boh(z^{-\alpha}), & \quad \alpha>0,\\
%            \Boh(\ln{z}), & \quad \alpha=0,\\
%            \Boh(1), & \quad -1 < \alpha <0.
%        \end{cases}
%    \end{equation}
%    \item [\rm (e)] As $z \to 1$, we have
%    \begin{equation}\label{1-S}
%        S (z) = \Boh(\ln{(z-1)}).
%    \end{equation}
    \end{itemize}
\end{rhp}

A close look at the $\lambda$-functions defined in \eqref{def-lambda} gives us the following estimates.
\begin{proposition}\label{re-lambda}
Let $\varepsilon$ be any fixed, small positive number, there exist positive constants $\mathsf{c_1}$ and $\mathsf{c_2}$ such that
\begin{align}
\Re{(\lambda_2(z) - \lambda_1(z))} < - \mathsf{c_1} |z|^{\frac23}, & \qquad z \in (\Sigma_1^{(1)} \cup \Sigma_5^{(1)}) \setminus D(1, \varepsilon),\label{re-2-1}\\
\Re{(\lambda_2(z) - \lambda_3(z))} < - \mathsf{c_2} |z|^{\frac23}, & \qquad z \in (\Sigma_2 \cup \Sigma_4)\setminus D(0, \varepsilon). \label{re-2-3}
\end{align}
for $s$ large enough, where the discs $D(1, \varepsilon)$ and $D(0,\varepsilon)$ are defined in \eqref{def:dz0r}.
\end{proposition}
\begin{proof}
Let
\begin{equation}\label{def-lambda*}
	\lambda_j^*(z): = \frac{3}{2^{1/3}}w_j(z)^2 - \frac{3}{2^{1/3}}w_j(z) - \frac{3}{2^{7/3}} , \qquad j=1,2,3.
\end{equation}
In view of Proposition \ref{prop:wfunction} and \eqref{def-lambda}, it is readily seen that
\begin{equation*}
|\lambda_j(z)-\lambda_j^*(z)|\leq \frac{\varrho}{s^{1/3}}|z|^{\frac23}, \qquad z\in \mathbb{C}\setminus D(0,\varepsilon),
\end{equation*}
for some positive $\varrho$. Thus, by the triangle inequality, it suffices to show \eqref{re-2-1} and \eqref{re-2-3} hold for $\lambda_j^*$.

We see from \eqref{def-w1} and \eqref{def-w2} that
\begin{align*}
\lambda_2^*(z) - \lambda_1^*(z)
&= \frac{3}{2^{1/3}} (w_2(z)^2 - w_2(z) -w_1(z)^2 + w_1(z))
\\
&=\frac{3}{4\cdot 2^{1/3}} \left(e^{\frac{2\pi \ii}{3}} \eta(z)^{\frac{2}{3}} + e^{-\frac{2\pi \ii}{3}} \eta(z)^{-\frac{2}{3}} - \eta(z)^{\frac{2}{3}} - \eta(z)^{-\frac{2}{3}}\right).
\end{align*}
For bounded $z\in (\Sigma_1^{(1)} \cup \Sigma_5^{(1)}) \setminus D(1, \varepsilon)$, by writing
$\eta (z) = r e^{\ii \theta}$, where  $r > 0$ and $\theta$ belongs to a compact subset of $(0,\pi)$, it follows that
\begin{align*}
\Re{(\lambda_2^*(z) - \lambda_1^*(z))} & = \frac{3}{2^{7/3}} (r^{\frac{2}{3}}  + r^{-\frac{2}{3}})\left(\cos\left(\frac{2(\pi + \theta)}{3}\right) - \cos\left(\frac{2 \theta}{3}\right)\right)\\
 & =  - \frac{3}{2^{4/3}} (r^{\frac{2}{3}}  + r^{-\frac{2}{3}})\sin\left(\frac{\pi + 2\theta}{3}\right) \sin\left(\frac{\pi}{3}\right)
 <  -\mathsf{c_1} |z|^{\frac 23},
\end{align*}
where $\mathsf{c_1}>0$ is independent of $z$. For large $z\in (\Sigma_1^{(1)} \cup \Sigma_5^{(1)}) \setminus D(1, \varepsilon)$, the above estimate follows from asymptotics of $\lambda_j^*$, which can be readily obtained from item (ii) of Proposition \ref{pro-lambda}.

The proof of \eqref{re-2-3} is similar, we omit the details here. This finishes the proof of Proposition \ref{re-lambda}.
\end{proof}
%with $\eta (z) = r e^{\ii \theta}$ and $r > 0, \theta \in (0, \pi)$. Thus,
%
%which shows \eqref{re-2-1}.
%
%Similarly, for $ z \in (\Sigma_2 \cup \Sigma_4)\setminus D(0, \varepsilon)$, by taking $s \to \infty$, we have
%\begin{equation}
%\begin{split}
%\Re{(\lambda_2(z) - \lambda_3(z))} \sim & \Re{\left(\frac{3}{4\cdot 2^{1/3}} \left(e^{\frac{2\pi \ii}{3}} \eta(z)^{\frac{2}{3}} + e^{-\frac{2\pi \ii}{3}} \eta(z)^{-\frac{2}{3}} - e^{-\frac{2\pi \ii}{3}}\eta(z)^{\frac{2}{3}} - e^{\frac{2\pi \ii}{3}}\eta(z)^{-\frac{2}{3}}\right)\right)}\\
%=&\frac{3}{2^{7/3}} (r^{\frac{2}{3}}  + r^{-\frac{2}{3}})\left(\cos{\frac{2(\pi + \theta)}{3}} - \cos{\frac{2 (\theta - \pi)}{3}}\right)\\
%= & - \frac{3}{2^{4/3}} (r^{\frac{2}{3}}  + r^{-\frac{2}{3}})\sin{\frac{2\theta}{3}} \sin{\frac{\pi}{3}}\\
%< & 0.
%\end{split}
%\end{equation}
The following corollary is an immediate consequence of Proposition \ref{re-lambda}.
\begin{corollary}\label{cor}
For $s$ large enough, there exists a constant $c>0$ such that
\begin{equation}
J_S(z)=I+\Boh(e^{-cs^{2/3}|z|^{2/3}}),
\end{equation}
uniformly for $z\in (\Sigma_1^{(1)} \cup \Sigma_5^{(1)}\cup \Sigma_2 \cup \Sigma_4)\setminus (D(0,\varepsilon)\cup D(1,\varepsilon))$.
\end{corollary}

\subsection{Global parametrix}\label{sec:RHP:global}
By Corollary \ref{cor}, we could ignore the jump of $S$ for $z$ bounded away from the intervals $(-\infty, 0) \cup (1, +\infty)$ and large $s$, which leads to the following global parametrix.
\begin{rhp}\label{rhp-N}
\hfill
\begin{itemize}
	\item [\rm (a)] $N_\alpha(z)$ is analytic for $z \in \mathbb{C} \setminus ((-\infty, 0]\cup [1, +\infty))$.
	\item [\rm (b)] For $x \in (-\infty, 0]\cup [1, +\infty)$, we have
	\begin{equation}\label{jump-N}
		N_{\alpha,+}(x) = N_{\alpha,-}(x) J_{N_\alpha}(x),
	\end{equation}
	where
	\begin{equation}\label{def:JNalpha}
		J_{N_\alpha}(x) = \begin{cases}
		\begin{pmatrix}
			1 & 0 & 0\\
			0 & 0 & -e^{-\alpha \pi \ii}\\
			0 & e^{-\alpha \pi \ii} & 0
		\end{pmatrix}, & \quad x \in (-\infty, 0),\\
		\begin{pmatrix}
			0 & 1 & 0\\
			-1 & 0 & 0\\
			0 & 0 & 1
		\end{pmatrix}, & \quad x \in (1, +\infty).
		\end{cases}
	\end{equation}
	\item [\rm (c)] As $z \to \infty$ and $\pm \Im{z} > 0$, we have
	\begin{equation}\label{infty-N}
		N_\alpha(z) = z^{-\frac{\alpha}{3}} \left(I +\Boh(z^{-1})\right)\diag{(z^{\frac{1}{3}}, 1, z^{-\frac{1}{3}})} L_{\pm} \diag{(e^{\pm \frac{\alpha \pi \ii}{3}}, e^{\mp \frac{\alpha \pi \ii}{3}}, 1)},
	\end{equation}
where the constant matrices $L_\pm$ are given in \eqref{def-L+-}.
\end{itemize}
\end{rhp}
The RH problem for $N_\alpha$ can be solved explicitly in two steps. As the first step, we construct a solution for the special case $\alpha=0$.
\begin{lemma}\label{prop:N0}
Let $w_i$, $i=1,2,3$, be three solutions of the algebraic equation \eqref{algebraic equation} given in \eqref{def-w1}--\eqref{def-w3}.
A solution of the RH problem \ref{rhp-N} with $\alpha=0$ is given by
\begin{equation}\label{def-tildeN}
	N_0(z) = \frac{1}{9}\begin{pmatrix}
		-5 \cdot 2^{-\frac 23} & -7 \cdot 2^{\frac 13} & 2^{-\frac 23}\\
		4 & -2 & -2\\
		-2^{\frac 53} & 2^{\frac 83} & -2^{\frac 53}
	\end{pmatrix}
	\begin{pmatrix}
		\mathcal{N}_1(w_1(z)) & \mathcal{N}_1(w_2(z)) & \mathcal{N}_1(w_3(z))\\
		\mathcal{N}_2(w_1(z)) & \mathcal{N}_2(w_2(z)) & \mathcal{N}_2(w_3(z))\\
		\mathcal{N}_3(w_1(z)) & \mathcal{N}_3(w_2(z)) &\mathcal{N}_3(w_3(z))
	\end{pmatrix},
\end{equation}
where
\begin{equation}\label{def-N123}
	\mathcal{N}_1(w) = \frac{w^2}{\sqrt{w(w-1)}}, \quad \mathcal{N}_2(w) = \frac{w(w-3/2)}{\sqrt{w(w-1)}}, \quad
    \mathcal{N}_3(w) = \frac{(w-3/2)^2}{\sqrt{w(w-1)}}.
\end{equation}
Here, the branch cut for the square root is taken along $\gamma_1^- \cup \gamma_2^-$, i.e., the curve defined by $w_{2,-}((-\infty, 0] \cup [1, +\infty))$; see Figure \ref{figure-w} for an illustration.
\end{lemma}
\begin{proof}
By item (i) of Proposition \ref{prop:wfunction} and the definition of $\mathcal N_j(w)$, $j = 1, 2, 3$, in \eqref{def-N123}, it is easily seen that if $x<0$
\begin{align*}
&\mathcal N_{j, +}(w_{1}(x)) = \mathcal N_{j}(w_{1, +}(x))=\mathcal N_j(w_{1, -}(x))=\mathcal N_{j,-}(w_{1}(x)),
\\
& \mathcal N_{j,+}(w_{2}(x)) = \mathcal N_j(w_{2, +}(x)) = \mathcal N_j(w_{3, -}(x))= \mathcal N_{j,-}(w_{3}(x)),
\\
& \mathcal N_{j,-}(w_{2}(x))=\mathcal N_j(w_{2, -}(x)) = - \mathcal N_j(w_{3, +}(x))= - \mathcal N_{j,+}(w_{3}(x)),
\end{align*}
and if $x>1$,
\begin{align*}\label{relation-N123}
& \mathcal N_{j,+}(w_{2}(x))=\mathcal N_j(w_{2, +}(x)) = \mathcal N_j(w_{1, -}(x))= \mathcal N_{j,-}(w_{1}(x)),
\\
& \mathcal N_{j,-}(w_{2}(x))=\mathcal N_j(w_{2, -}(x)) = -\mathcal N_j(w_{1, +}(x))= -\mathcal N_{j,+}(w_{1}(x)),
\\
& \mathcal N_{j,+}(w_{3}(x)) =\mathcal N_j(w_{3, +}(x)) = \mathcal N_j(w_{3, -}(x))= \mathcal N_{j,-}(w_{3}(x)),
\end{align*}
These are exactly the jump condition \eqref{jump-N} and \eqref{def:JNalpha} with $\alpha=0$.

To show the asymptotic condition \eqref{infty-N} with $\alpha=0$, we obtain from items (i) and (ii) of Proposition \ref{prop:wfunction}, \eqref{def-N123} and straightforward calculations that, as $z\to \infty$,
\begin{align*}
& \mathcal N_{1}(w_{1}(z))= \begin{cases}
-2^{-\frac 13} \omega z^{\frac 13} + 1 -\frac{5 \cdot 2^{1/3} \omega^2}{8}z^{-\frac 13} + \frac{5 \cdot 2^{2/3} \omega}{24} z^{-\frac 23}
\\
 \quad+\frac{5}{64} z^{-1} - \frac{35 \cdot 2^{1/3} \omega^2}{192} z^{-\frac 43} + \Boh (z^{-\frac 53}), &\qquad \Im{z}>0,\\
-2^{-\frac 13} \omega^2 z^{\frac 13} + 1 -\frac{5 \cdot 2^{1/3} \omega}{8}z^{-\frac 13} + \frac{5 \cdot 2^{2/3} \omega^2}{24} z^{-\frac 23}
\\
\quad+\frac{5}{64} z^{-1} - \frac{35 \cdot 2^{1/3} \omega}{192} z^{-\frac 43} + \Boh (z^{-\frac 53}), &\qquad \Im{z}<0,
\end{cases}\\
& \mathcal N_{1}(w_{2}(z))= \begin{cases}
-2^{-\frac 13} \omega^2 z^{\frac 13} + 1 -\frac{5 \cdot 2^{1/3} \omega}{8}z^{-\frac 13} + \frac{5 \cdot 2^{2/3} \omega^2}{24} z^{-\frac 23}
\\
\quad+\frac{5}{64} z^{-1} - \frac{35 \cdot 2^{1/3} \omega}{192} z^{-\frac 43} + \Boh (z^{-\frac 53}), &\qquad \Im{z}>0,\\
-2^{-\frac 13} \omega z^{\frac 13} + 1 -\frac{5 \cdot 2^{1/3} \omega^2}{8}z^{-\frac 13} + \frac{5 \cdot 2^{2/3} \omega}{24} z^{-\frac 23}
\\
 \quad+\frac{5}{64} z^{-1} - \frac{35 \cdot 2^{1/3} \omega^2}{192} z^{-\frac 43} + \Boh (z^{-\frac 53}), &\qquad \Im{z}<0,
\end{cases}\\
& \mathcal N_{1}(w_{3}(z))= -2^{-\frac 13} z^{\frac 13} + 1 -\frac{5 \cdot 2^{1/3}}{8}z^{-\frac 13} + \frac{5 \cdot 2^{2/3}}{24} z^{-\frac 23} +\frac{5}{64} z^{-1} - \frac{35 \cdot 2^{1/3}}{192} z^{-\frac 43} + \Boh (z^{-\frac 53}),\\
& \mathcal N_{2}(w_{1}(z))= \begin{cases}
-2^{-\frac 13} \omega z^{\frac 13} -\frac 12 +\frac{2^{1/3} \omega^2}{8}z^{-\frac 13} + \frac{2^{2/3} \omega}{48} z^{-\frac 23}
\\
\quad-\frac{7}{64} z^{-1} + \frac{23 \cdot 2^{1/3} \omega^2}{384} z^{-\frac 43} + \Boh (z^{-\frac 53}), &\qquad \Im{z}>0,\\
-2^{-\frac 13} \omega^2 z^{\frac 13} -\frac 12 +\frac{2^{1/3} \omega}{8}z^{-\frac 13} + \frac{5 \cdot 2^{2/3} \omega^2}{48} z^{-\frac 23}
\\
\quad-\frac{7}{64} z^{-1} + \frac{23 \cdot 2^{1/3} \omega}{384} z^{-\frac 43} + \Boh (z^{-\frac 53}), & \qquad \Im{z}<0,
\end{cases}\\
& \mathcal N_{2}(w_{2}(z))= \begin{cases}
-2^{-\frac 13} \omega^2 z^{\frac 13} -\frac 12 +\frac{2^{1/3} \omega}{8}z^{-\frac 13} + \frac{5 \cdot 2^{2/3} \omega^2}{48} z^{-\frac 23}
\\
\quad-\frac{7}{64} z^{-1} + \frac{23 \cdot 2^{1/3} \omega}{384} z^{-\frac 43} + \Boh (z^{-\frac 53}), &\qquad \Im{z}>0,\\
-2^{-\frac 13} \omega z^{\frac 13} -\frac 12 +\frac{2^{1/3} \omega^2}{8}z^{-\frac 13} + \frac{2^{2/3} \omega}{48} z^{-\frac 23}
\\
\quad-\frac{7}{64} z^{-1} + \frac{23 \cdot 2^{1/3} \omega^2}{384} z^{-\frac 43} + \Boh (z^{-\frac 53}), & \qquad \Im{z}<0,
\end{cases}\\
& \mathcal N_{2}(w_{3}(z))= -2^{-\frac 13}  z^{\frac 13} -\frac 12 +\frac{2^{1/3}}{8}z^{-\frac 13} + \frac{5 \cdot 2^{2/3}}{48} z^{-\frac 23} -\frac{7}{64} z^{-1} + \frac{23 \cdot 2^{1/3}}{384} z^{-\frac 43} + \Boh (z^{-\frac 53}),
\end{align*}
and
\begin{align*}
& \mathcal N_{3}(w_{1}(z))= \begin{cases}
-2^{-\frac 13} \omega z^{\frac 13} -2 -\frac{11 \cdot 2^{1/3} \omega^2}{8}z^{-\frac 13} + \frac{2^{2/3} \omega}{6} z^{-\frac 23}
\\
\quad+\frac{17}{64} z^{-1} - \frac{7 \cdot 2^{1/3} \omega^2}{96} z^{-\frac 43} + \Boh (z^{-\frac 53}), &\qquad \Im{z}>0,\\
-2^{-\frac 13} \omega^2 z^{\frac 13} -2 -\frac{11 \cdot 2^{1/3} \omega}{8}z^{-\frac 13} + \frac{2^{2/3} \omega^2}{6} z^{-\frac 23}
\\
\quad+\frac{17}{64} z^{-1} - \frac{7 \cdot 2^{1/3} \omega}{96} z^{-\frac 43} + \Boh (z^{-\frac 53}), &\qquad \Im{z}<0,
\end{cases}\\
& \mathcal N_{3}(w_{2}(z))= \begin{cases}
-2^{-\frac 13} \omega^2 z^{\frac 13} -2 -\frac{11 \cdot 2^{1/3} \omega}{8}z^{-\frac 13} + \frac{2^{2/3} \omega^2}{6} z^{-\frac 23}
\\
\quad+\frac{17}{64} z^{-1} - \frac{7 \cdot 2^{1/3} \omega}{96} z^{-\frac 43} + \Boh (z^{-\frac 53}), &\qquad \Im{z}>0,\\
-2^{-\frac 13} \omega z^{\frac 13} -2 -\frac{11 \cdot 2^{1/3} \omega^2}{8}z^{-\frac 13} + \frac{2^{2/3} \omega}{6} z^{-\frac 23}
\\
\quad+\frac{17}{64} z^{-1} - \frac{7 \cdot 2^{1/3} \omega^2}{96} z^{-\frac 43} + \Boh (z^{-\frac 53}), &\qquad \Im{z}<0,
\end{cases}\\
& \mathcal N_{3}(w_{3}(z))= -2^{-\frac 13} z^{\frac 13} -2 -\frac{11 \cdot 2^{1/3}}{8}z^{-\frac 13} + \frac{2^{2/3}}{6} z^{-\frac 23} +\frac{17}{64} z^{-1} - \frac{7 \cdot 2^{1/3}}{96} z^{-\frac 43} + \Boh (z^{-\frac 53}).
\end{align*}
Inserting the above formulas into \eqref{def-tildeN}, we have
\begin{equation}\label{eq:N0asy}
		N_0(z) = \left(I +\Boh(z^{-1})\right)\diag{(z^{\frac{1}{3}}, 1, z^{-\frac{1}{3}})} L_{\pm}, \qquad \pm \Im{z} > 0,
\end{equation}
as required.

This completes the proof of Lemma \ref{prop:N0}.
\end{proof}

%Let $\widetilde{N}(z)$ denotes the solution to this RH problem in the case $\alpha = 0$. Then $\widetilde{N}(z)$ satisfies the following properties:
%\begin{itemize}
%	\item [\rm (a)] $\widetilde{N}(z)$ is analytic for $z \in \mathbb{C} \setminus (-\infty, 0]\cup [1, +\infty)$.
%	\item [\rm (b)] For $x \in (-\infty, 0]\cup [1, +\infty)$, we have
%	\begin{equation}\label{jump-tildeN}
%		\widetilde{N}_{+}(x) = \widetilde{N}_{-}(x) J_{\widetilde{N}}(x),
%	\end{equation}
%	where
%	\begin{equation}
%		J_{\widetilde{N}}(x) = \begin{cases}
%		\begin{pmatrix}
%			1 & 0 & 0\\
%			0 & 0 & -1\\
%			0 & 1 & 0
%		\end{pmatrix}, & \quad x \in (-\infty, 0),\\
%		\begin{pmatrix}
%			0 & 1 & 0\\
%			-1 & 0 & 0\\
%			0 & 0 & 1
%		\end{pmatrix}, & \quad x \in (1, +\infty).
%		\end{cases}
%	\end{equation}
%	\item [\rm (c)] As $z \to \infty$, we have
%	\begin{equation}
%		\widetilde{N}(z) = \left(I +\Boh(z^{-1})\right)\diag{(z^{\frac{1}{3}}, 1, z^{-\frac{1}{3}})} L_{\pm}
%	\end{equation}
%\end{itemize}

For general $\alpha \neq 0$, we define, with the aid of $N_0$ in \eqref{def-tildeN},
\begin{equation}\label{def-N}
	N_{\alpha}(z) = C_{\alpha} N_0(z) \diag(D_1(z),D_2(z),D_3(z)),
\end{equation}
where
\begin{equation}\label{def:C-alpha}
C_{\alpha} = 2^{-\frac{\alpha}{3}} \begin{pmatrix}
1 & - 2^{-\frac 23}\alpha &  2^{-\frac 73}\alpha (\alpha +1)\\
0 & 1 & -2^{-\frac 23}\alpha\\
0 & 0 & 1
\end{pmatrix}
\end{equation}
and
\begin{equation}\label{def:di}
D_1(z)=w_1(z)^{-\alpha},\qquad D_2(z)=w_2(z)^{-\alpha}, \qquad D_3(z)=e^{\alpha \pi \ii}w_3(z)^{-\alpha},
\end{equation}
and the branch cut for $z^\alpha$ is taken along $\gamma_1^{-} = w_{2, -} (-\infty,0)$. The functions $D_i(z)$, $i=1,2,3$, can be viewed as an analogue to the Szeg\H{o} function; cf. \cite{Kuijlaars2004}.
\begin{lemma}\label{pro-N}
The function $N_\alpha(z)$ defined in \eqref{def-N} solves the RH problem \ref{rhp-N}.
\end{lemma}
\begin{proof}
From the definitions \eqref{def:di} with special choice of the branch cut, it is readily seen that
		\begin{equation}\label{jump-D0}
			D_{1,+} (x)= D_{1,-}(x), \qquad
			D_{2,+} (x)= e^{-\alpha \pi \ii}D_{3,-}(x), \qquad
			D_{3,+} (x)= e^{-\alpha \pi \ii}D_{2,-}(x),
		\end{equation}
for $x \in (-\infty, 0)$, and
		\begin{equation}\label{jupm-D1}
			D_{1,+} (x)= D_{2,-}(x), \qquad D_{2,+} (x)= D_{1,-}(x), \qquad D_{3,+} (x)= D_{3,-}(x).
		\end{equation}
for $x \in (1, +\infty)$. A combination of \eqref{jump-N} and \eqref{def:JNalpha} with $\alpha=0$, \eqref{def-N} and the above relations implies that $N_\alpha(z)$ is indeed analytic in $\mathbb{C} \setminus ((-\infty, 0]\cup [1, +\infty))$ and satisfies the jump condition \eqref{jump-N} and \eqref{def:JNalpha}.

In view of \eqref{def:di} and the asymptotic behaviors of the $w$-functions given in \eqref{infty-w2} and \eqref{infty-w3}, we have, as $z \to \infty $,
\begin{align}
& D_1(z) = \begin{cases}
e^{\frac{\alpha \pi \ii}{3}}2^{\frac{\alpha}{3}}z^{-\frac{\alpha}{3}}\left(1 + \frac{\alpha \omega^2}{2^{2/3}}z^{-\frac 13} + \frac{\alpha (\alpha-1) \omega}{2^{7/3}}z^{-\frac 23} + \frac{\alpha(\alpha^2 -3\alpha + 2^{4/3} -4)}{24}z^{-1}\right.\\
\quad\left.+ \frac{\alpha(\alpha + 1)(\alpha^2 - 7 \alpha + 2^{10/3} -6) \omega^2}{96 \cdot 2^{2/3}} z^{-\frac 43} + \Boh(z^{-\frac 53})\right), & \Im{z} > 0,\\
e^{-\frac{\alpha \pi \ii}{3}}2^{\frac{\alpha}{3}}z^{-\frac{\alpha}{3}}\left(1 + \frac{\alpha \omega}{2^{2/3}}z^{-\frac 13} + \frac{\alpha (\alpha-1) \omega^2}{2^{7/3}}z^{-\frac 23} + \frac{\alpha(\alpha^2 -3\alpha + 2^{4/3} -4)}{24}z^{-1}\right.\\
\quad\left.+ \frac{\alpha(\alpha + 1)(\alpha^2 - 7 \alpha + 2^{10/3} -6) \omega}{96 \cdot 2^{2/3}} z^{-\frac 43} + \Boh(z^{-\frac 53})\right), & \Im{z} < 0,
\end{cases}\label{D1}\\
&D_2(z) = \begin{cases}
e^{-\frac{\alpha \pi \ii}{3}}2^{\frac{\alpha}{3}}z^{-\frac{\alpha}{3}}\left(1 + \frac{\alpha \omega}{2^{2/3}}z^{-\frac 13} + \frac{\alpha (\alpha-1) \omega^2}{2^{7/3}}z^{-\frac 23} + \frac{\alpha(\alpha^2 -3\alpha + 2^{4/3} -4)}{24}z^{-1}\right.\\
\quad\left.+ \frac{\alpha(\alpha + 1)(\alpha^2 - 7 \alpha + 2^{10/3} -6) \omega}{96 \cdot 2^{2/3}} z^{-\frac 43} + \Boh(z^{-\frac 53})\right), & \Im{z} > 0,\\
e^{\frac{\alpha \pi \ii}{3}}2^{\frac{\alpha}{3}}z^{-\frac{\alpha}{3}}\left(1 + \frac{\alpha \omega^2}{2^{2/3}}z^{-\frac 13} + \frac{\alpha (\alpha-1) \omega}{2^{7/3}}z^{-\frac 23} + \frac{\alpha(\alpha^2 -3\alpha + 2^{4/3} -4)}{24}z^{-1}\right.\\
\quad\left.+ \frac{\alpha(\alpha + 1)(\alpha^2 - 7 \alpha + 2^{10/3} -6) \omega^2}{96 \cdot 2^{2/3}} z^{-\frac 43} + \Boh(z^{-\frac 53})\right), & \Im{z} < 0,
\end{cases}\label{D2}\\
& D_3(z) = 2^{\frac{\alpha}{3}}z^{-\frac{\alpha}{3}}\left(1 + \frac{\alpha}{2^{2/3}}z^{-\frac 13} + \frac{\alpha (\alpha-1)}{2^{7/3}}z^{-\frac 23} + \frac{\alpha(\alpha^2 -3\alpha + 2^{4/3} -4)}{24}z^{-1}\right.\nonumber\\
&\qquad \qquad \left.+ \frac{\alpha(\alpha + 1)(\alpha^2 - 7 \alpha + 2^{10/3} -6)}{96 \cdot 2^{2/3}} z^{-\frac 43} + \Boh(z^{-\frac 53})\right).\label{D3}
\end{align}
%\begin{align}
%\diag(D_1(z),D_2(z),D_3(z)) & = 2^{\frac{\alpha}{3}}z^{-\frac{\alpha}{3}}\left(I+  \frac{\alpha}{2^{2/3}z^{1/3}} \diag(\omega^2, \omega, 1) + \frac{\alpha (\alpha-1)}{2^{7/3} z^{2/3}}\diag(\omega, \omega^2, 1)\right. \nonumber\\
%& \left. \quad +\frac{\alpha(\alpha^2 -3\alpha + 2^{4/3} -4)}{24z}I + \right.\\
%& \quad \diag(e^{\frac{\alpha \pi \ii}{3}},e^{-\frac{\alpha \pi \ii}{3}},1)
%\end{align}
Inserting the above equations into \eqref{def-N}, together with \eqref{eq:N0asy}, gives us, if $\Im{z}>0$,
\begin{align}
N_\alpha(z) & = \left(I +  \Boh(z^{-1})\right)2^{\frac{\alpha}{3}}z^{-\frac{\alpha}{3}}C_{\alpha} \diag{(z^{\frac{1}{3}}, 1, z^{-\frac{1}{3}})} L_{+}\left(I +
\frac{\alpha}{2^{2/3}}z^{-\frac 13} \diag(\omega^2, \omega, 1)\right.\nonumber\\
&\quad+\frac{\alpha (\alpha-1)}{2^{7/3}}z^{-\frac 23} \diag(\omega, \omega^2, 1)+\frac{\alpha(\alpha^2 -3\alpha + 2^{4/3} -4)}{24}z^{-1} \diag(1,1,1)\nonumber\\
&\quad\left. + \frac{\alpha(\alpha + 1)(\alpha^2 - 7 \alpha + 2^{10/3} -6)}{96 \cdot 2^{2/3}} z^{-\frac 43}\diag(\omega^2, \omega, 1) +\Boh(z^{-\frac 53})\right) \diag{(e^{\frac{\alpha \pi \ii}{3}}, e^{-\frac{\alpha \pi \ii}{3}}, 1)},
\end{align}
and if $\Im{z}<0$,
\begin{align}
N_\alpha(z) & = \left(I +  \Boh(z^{-1})\right)2^{\frac{\alpha}{3}}z^{-\frac{\alpha}{3}}C_{\alpha} \diag{(z^{\frac{1}{3}}, 1, z^{-\frac{1}{3}})} L_{-}\left(I +
\frac{\alpha}{2^{2/3}}z^{-\frac 13} \diag(\omega, \omega^2, 1)\right.\nonumber\\
&\quad+\frac{\alpha (\alpha-1)}{2^{7/3}}z^{-\frac 23} \diag(\omega^2, \omega, 1)+\frac{\alpha(\alpha^2 -3\alpha + 2^{4/3} -4)}{24}z^{-1} \diag(1,1,1)\nonumber\\
&\quad\left. + \frac{\alpha(\alpha + 1)(\alpha^2 - 7 \alpha + 2^{10/3} -6)}{96 \cdot 2^{2/3}} z^{-\frac 43}\diag(\omega, \omega^2, 1) +\Boh(z^{-\frac 53})\right) \diag{(e^{-\frac{\alpha \pi \ii}{3}}, e^{\frac{\alpha \pi \ii}{3}}, 1)},
\end{align}
where $C_{\alpha}$ and $L_{\pm}$ are given in \eqref{def:C-alpha} and \eqref{def-L+-}.

After a straightforward calculation, we have
\begin{equation}\label{infty-N1}
	N_\alpha(z) = \left(I + \frac{\mathsf{N}_1}{z} + \Boh(z^{-2})\right)z^{-\frac{\alpha}{3}} \diag{(z^{\frac{1}{3}}, 1, z^{-\frac{1}{3}})} L_{\pm} \diag{(e^{\pm \frac{\alpha \pi \ii}{3}}, e^{\mp \frac{\alpha \pi \ii}{3}}, 1)},\quad \pm \Im{z} >0,
	\end{equation}
where
	\begin{equation}\label{def-N1}
	\mathsf{N}_1 = \begin{pmatrix}
	\ast & \ast & \ast\\
	\ast & \ast & \ast\\
	\alpha/2^{\frac23} & \ast & \ast
	\end{pmatrix},
	\end{equation}
as shown in \eqref{infty-N}.
%with $\ast$ are some unimportant entry and the constant matrices $L_{\pm}$ are given in \eqref{def-L+-}.

This completes the proof of Lemma \ref{pro-N}.
\end{proof}
Finally, from the asymptotic behaviors of the $w$-functions given in items (iii) and (iv) of Proposition \ref{prop:wfunction}, it is readily seen the following proposition regarding the refined asymptotic behaviors of the global parametrix $N_\alpha(z)$ near $z=0$ and $z=1$.
\begin{proposition}
	With $N_\alpha(z)$ defined in \eqref{def-N}, we have, as $z \to 0$,
	\begin{align}\label{N-0}
	N_\alpha(z) &= \frac{C_{\alpha}}{9}\begin{pmatrix}
		-5 \cdot 2^{-\frac 23} & -7 \cdot 2^{\frac 13} & 2^{-\frac 23}\\
		4 & -2 & -2\\
		-2^{\frac 53} & 2^{\frac 83} & -2^{\frac 53}
	\end{pmatrix}\left[z^{-\frac{1}{4}} \begin{pmatrix}
	0 & 0 & 0\\
	0 & 0 & 0\\
	0 & - \ii \frac{3^{9/4}}{4} & -\frac{3^{9/4}}{4}
	\end{pmatrix} +  \begin{pmatrix}
	\frac{3 \sqrt{3}}{2} & 0 & 0\\
	0 & 0 & 0\\
	0 & 0 & 0
	\end{pmatrix}\right.
     \nonumber \\
	 &\left. ~~ + z^{\frac{1}{4}}  \begin{pmatrix}
	0 & 0 & 0\\
	0 & \ii \frac{3^{3/4}}{2} & -\frac{3^{3/4}}{2}\\
	0 & \ii \frac{3^{7/4}}{4} & -\frac{3^{7/4}}{4}
	\end{pmatrix} + \Boh(z^{\frac{3}{4}})\right] \diag{\left(\left(\frac{2}{3}\right)^{\alpha}, 3^{\frac{\alpha}{2}}z^{-\frac{\alpha}{2}}, 3^{\frac{\alpha}{2}}z^{-\frac{\alpha}{2}}\right)}.
	\end{align}
If $z \to 1$ and $\Im{z}>0$, we have
	\begin{align}\label{N1}
    N_\alpha(z) &= \frac{C_{\alpha}}{9}\begin{pmatrix}
		-5 \cdot 2^{-\frac 23} & -7 \cdot 2^{\frac 13} & 2^{-\frac 23}\\
		4 & -2 & -2\\
		-2^{\frac 53} & 2^{\frac 83} & -2^{\frac 53}
	\end{pmatrix}\left[3^{\frac{1}{4}}e^{\frac{\pi \ii}{4}}(z-1)^{-\frac{1}{4}}\begin{pmatrix}
	1 & -\ii & 0\\
	-\frac{1}{2} &  \frac{\ii}{2} & 0\\
	\frac{1}{4} & -\frac{\ii}{4} & 0
	\end{pmatrix} \right.
    \nonumber \\
   & \left. ~~  - \frac{2^{\alpha}}{\sqrt{3}} \begin{pmatrix}
	0 & 0 & \frac{1}{2}\\
	0 & 0 & 2\\
	0 & 0 & 8
	\end{pmatrix}
	+  3^{-\frac{1}{4}} e^{\frac{\pi \ii}{4}} (z-1)^{\frac{1}{4}}
	\begin{pmatrix}
	\ii \left(\alpha -\frac{5}{3}\right) & -\left(\alpha -\frac{5}{3}\right) & 0\\
	-\ii \left(\frac{\alpha}{2}+\frac{2}{3}\right) & \frac{\alpha}{2}+\frac{2}{3} & 0\\
	\ii \left(\frac{\alpha}{4}+\frac{13}{12}\right) & -\left(\frac{\alpha}{4}+\frac{13}{12}\right) & 0
	\end{pmatrix} \right.
   \nonumber
\\
	& \left. ~~ +\frac{3^{1/4}}{9}e^{\frac{\pi \ii}{4}}(z-1)^{\frac{3}{4}}
	\begin{pmatrix}
	-\frac{\alpha(3\alpha +5)}{2} -\frac{1}{12} & \ii \left(-\frac{\alpha(3\alpha +5)}{2} -\frac{1}{12}\right) & 0\\
	\frac{\alpha(3\alpha +13)}{4}-\frac{35}{24} & -\ii \left(\frac{\alpha(3\alpha +13)}{4} -\frac{35}{24}\right) & 0\\
	-\frac{\alpha(3\alpha +31)}{8} -\frac{253}{48} & \ii \left(\frac{\alpha(3\alpha +31)}{8} +\frac{253}{48}\right) & 0
	\end{pmatrix} \right.
\nonumber
\\
	&  \left. ~~ + \frac{2^{\alpha}}{9\sqrt{3}}(z-1)
	\begin{pmatrix}
	0 & 0 & 2\alpha-\frac{8}{3}\\
	0 & 0 & 8\alpha-\frac{14}{3}\\
	0 & 0 & 32\alpha+\frac{16}{3}
	\end{pmatrix}
     \nonumber \right.
\\
	& \left. ~~+ \frac{3^{-1/4}}{54}e^{\frac{\pi \ii}{4}}(z-1)^{\frac{5}{4}}
	\begin{pmatrix}
	-\ii \left(3\alpha^3 - \frac{15 \alpha}{2} - \frac{1}{2}\right) & 3\alpha^3 -\frac{15 \alpha}{2} - \frac{1}{2} & 0\\
	\ii \left(\frac{3\alpha^3}{2} + \frac{27 \alpha^2}{2} + \frac{39\alpha}{4}-7 \right) & -\left(\frac{3\alpha^3}{2} + \frac{27 \alpha^2}{2} + \frac{39\alpha}{4}-7 \right) & 0\\
	-\ii \frac{6 \alpha^3 + 108\alpha^2 + 417 \alpha + 269}{8} & \frac{6 \alpha^3 + 108\alpha^2 + 417 \alpha + 269}{8} & 0
	\end{pmatrix} \right.
\nonumber
\\
	& \left. ~~  + \Boh((z-1)^{\frac{7}{4}}) \right].
\end{align}
\end{proposition}

Since the jump matrices for $S$ and $N_\alpha$ are not uniformly close to each other near $z=0$ and $z=1$, we next construct local parametrices at these two points, respectively.

%------------------------------------------------------------------------------------------------------------------------------------------------
\subsection{Local parametrices near $z=0$ and $z=1$}
In a small disc $D(0,\varepsilon)$ centered at $0$, we seek a $2 \times 2$ matrix-valued function $P^{(0)}(z)$ satisfying an RH problem as follows.
\begin{rhp}\label{rhp-p0}
\hfill
\begin{itemize}
	\item [\rm (a)] $P^{(0)}(z)$ is analytic in $D(0, \varepsilon) \setminus \Sigma_T$, where $\Sigma_T$ is defined in \eqref{def:sigmaT}.
	\item [\rm (b)] For $z \in D(0, \varepsilon) \cap \Sigma_T$, we have
	\begin{equation}\label{jump-P0}
		P^{(0)}_+(z) = P^{(0)}_-(z) \begin{cases}
			\begin{pmatrix}
	     			1 & 0 & 0\\
	     			0 & 1 & e^{\alpha \pi \ii}e^{s^{2/3}(\lambda_{2}(z)-\lambda_{3}(z))}\\
	     			0 & 0 & 1
			\end{pmatrix}, &\quad z \in D(0, \varepsilon) \cap\Sigma_2,\\
			\begin{pmatrix}
	     			1 & 0 & 0\\
	     			0 & 0 & -e^{-\alpha \pi \ii}\\
	     			0 & e^{-\alpha \pi \ii} & 0
			\end{pmatrix}, &\quad z \in D(0, \varepsilon) \cap \Sigma_3,\\
			\begin{pmatrix}
	     			1 & 0 & 0\\
	     			0 & 1 & e^{-\alpha \pi \ii}e^{s^{2/3}(\lambda_{2}(z)-\lambda_{3}(z))}\\
	     			0 & 0 & 1
			\end{pmatrix}, &\quad z \in D(0, \varepsilon) \cap \Sigma_4.
		\end{cases}
	\end{equation}
	\item [\rm (c)] As $s \to \infty$, we have the matching condition
	\begin{equation}\label{match-P0}
		P^{(0)}(z) = \left(I + \Boh(s^{-\frac{2}{3}})\right) N_\alpha(z), \qquad z\in \partial D(0,\varepsilon),
	\end{equation}
	where $N_\alpha(z)$ is given in \eqref{def-N}.
\end{itemize}
\end{rhp}

The RH problem \ref{rhp-p0} for $P^{(0)}(z)$ can be solved explicitly with the aid of the Bessel parametrix $\Phi_{\alpha}^{(\text{Bes})}$ described in Appendix \ref{A}. To this aim, we introduce the local conformal mapping
\begin{equation}\label{def-tildef}
	f(z) = \frac{1}{4}(\lambda_2(z)-\lambda_3(z))^2 = c_1^2 z + \Boh(z^2), \qquad z\to 0,
\end{equation}
where $c_1$ is given in \eqref{tildec}; see \eqref{0-lambda2} and \eqref{0-lambda3}. We then define
\begin{align}\label{def-P0}
	P^{(0)}(z) & =  E(z) \diag \left(1, f(z)^{-\frac{\alpha}{2}}, f(z)^{-\frac{\alpha}{2}} \right)
	\begin{pmatrix}
		1 & 0 & 0\\
		0 & \left(\Phi_{\alpha}^{(\text{Bes})}\right)_{11}(s^{\frac{4}{3}}f(z)) & \left(\Phi_{\alpha}^{(\text{Bes})}\right)_{12}(s^{\frac{4}{3}}f(z))\\
		0 & \left(\Phi_{\alpha}^{(\text{Bes})}\right)_{21}(s^{\frac{4}{3}}f(z)) & \left(\Phi_{\alpha}^{(\text{Bes})}\right)_{22}(s^{\frac{4}{3}}f(z))
		\end{pmatrix}
        \nonumber
        \\
		&~~ \times \diag\left(1, e^{-\frac{s^{1/3}(\lambda_2(z)-\lambda_3(z))}{2}},e^{\frac{s^{1/3}(\lambda_2(z)-\lambda_3(z))}{2}}\right),
\end{align}
where
\begin{equation}\label{def-E0}
	E(z) = \frac{ N_\alpha(z)}{\sqrt{2}}
	\begin{pmatrix}
		\sqrt{2} & 0 & 0\\
		0 & -\ii \pi^{\frac{1}{2}} s^{\frac{1}{3}} f(z)^{\frac{1}{4}} & \pi^{-\frac{1}{2}} s^{-\frac{1}{3}} f(z)^{-\frac{1}{4}}\\
		0 & \pi^{\frac{1}{2}} s^{\frac{1}{3}} f(z)^{\frac{1}{4}} & -\ii \pi^{-\frac{1}{2}} s^{-\frac{1}{3}} f(z)^{-\frac{1}{4}}
	\end{pmatrix}
	\diag \left(1, f(z)^{\frac{\alpha}{2}}, f(z)^{\frac{\alpha}{2}} \right),
\end{equation}
and $\Phi_{\alpha}^{(\text{Bes})}$ solves the RH problem \ref{rhp:Bessel}.
\begin{lemma}\label{lem-P0}
The matrix-valued function $P^{(0)}(z)$ defined in \eqref{def-P0} solves the RH problem \ref{rhp-p0}.
\end{lemma}
\begin{proof}
We first show the analyticity of $E(z)$ near $z=0$. According to its definition in \eqref{def-E0}, the possible jump is on
$(- \varepsilon ,0)$. It follows from \eqref{jump-N} and \eqref{def-tildef} that, if $z\in (-\varepsilon, 0)$,
\begin{align}
&E_-(z)^{-1}E_+(z)
\nonumber
\\
&= \frac{1}{2}\diag\left(1, f_-(z)^{-\frac{\alpha}{2}}, f_-(z)^{-\frac{\alpha}{2}} \right)
	\begin{pmatrix}
		\sqrt{2} & 0 & 0\\
		0 & \ii \pi^{-\frac{1}{2}} s^{-\frac{1}{3}} f_-(z)^{-\frac{1}{4}} & \pi^{-\frac{1}{2}} s^{-\frac{1}{3}} f_-(z)^{-\frac{1}{4}}\\
		0 & \pi^{\frac{1}{2}} s^{\frac{1}{3}} f_-(z)^{\frac{1}{4}} & -\ii \pi^{\frac{1}{2}} s^{\frac{1}{3}} f_-(z)^{\frac{1}{4}}
	\end{pmatrix}
\nonumber
\\
	&~~ \times \begin{pmatrix}
			1 & 0 & 0\\
			0 & 0 & -e^{-\alpha \pi \ii}\\
			0 & e^{-\alpha \pi \ii} & 0
		\end{pmatrix}
		\begin{pmatrix}
		\sqrt{2} & 0 & 0\\
		0 & -\ii \pi^{\frac{1}{2}} s^{\frac{1}{3}} f_+(z)^{\frac{1}{4}} & \pi^{-\frac{1}{2}} s^{-\frac{1}{3}} f_+(z)^{-\frac{1}{4}}\\
		0 & \pi^{\frac{1}{2}} s^{\frac{1}{3}} f_+(z)^{\frac{1}{4}} & -\ii \pi^{-\frac{1}{2}} s^{-\frac{1}{3}} f_+(z)^{-\frac{1}{4}}
	\end{pmatrix}
\nonumber \\
	&~~ \times
\diag \left(1,f_+(z)^{\frac{\alpha}{2}},f_+(z)^{\frac{\alpha}{2}} \right)
	\nonumber
    \\
	& = \diag\left(1, f_-(z)^{-\frac{\alpha}{2}} e^{-\alpha \pi \ii}f_+(z)^{\frac{\alpha}{2}}, f_-(z)^{-\frac{\alpha}{2}} e^{-\alpha \pi \ii}f_+(z)^{\frac{\alpha}{2}}\right)=I.
\end{align}
Moreover, we see from \eqref{N-0} and \eqref{def-tildef} that
\begin{align}
E(0) &= \frac{C_{\alpha}}{9 \sqrt{2}}\begin{pmatrix}
		-5 \cdot 2^{-\frac 23} & -7 \cdot 2^{\frac 13} & 2^{-\frac 23}\\
		4 & -2 & -2\\
		-2^{\frac 53} & 2^{\frac 83} & -2^{\frac 53}
	\end{pmatrix}
    \nonumber \\
	&~~ \times\begin{pmatrix}
	2^{\alpha - \frac{1}{2}} 3^{\frac{3}{2} - \alpha} & 0 & 0\\
	0 & 0 & \ii 3^{\frac{3}{4} +\frac{\alpha}{2}}|c_1|^{\alpha-\frac{1}{2}} \pi^{-\frac{1}{2}} s^{-\frac{1}{3}}\\
	0 &  -\frac{3^{9/4 +\alpha/2}}{2} |c_1|^{\alpha+\frac{1}{2}} \pi^{\frac{1}{2}} s^{\frac{1}{3}}  & \ii \frac{3^{7/4 +\alpha/2}}{2} |c_1|^{\alpha-\frac{1}{2}} \pi^{-\frac{1}{2}} s^{-\frac{1}{3}}
	\end{pmatrix},
	\end{align}
where $C_\alpha$ and $c_1$ are given in \eqref{def:C-alpha} and \eqref{tildec}, respectively. Thus, $E(z)$ is indeed analytic in $D(0, \varepsilon)$. It is then straightforward to verify $P^{(0)}(z)$ satisfies the jump condition \eqref{jump-P0} by using \eqref{jump-Bessel}, item (i) of Proposition \ref{pro-de} and the analyticity of $E(z)$.

It remains to check the matching condition \eqref{match-P0}. As $s \to \infty$, applying \eqref{def-P0}, \eqref{def-E0} and the asymptotic behavior of the Bessel parametrix $\Phi_{\alpha}^{(\text{Bes})}(z)$ at infinity in \eqref{infty-Bessel} yields
\begin{equation}
P^{(0)}(z) N_\alpha(z)^{-1} =  I + \frac{J_1(z)}{ s^{2/3}}+ \Boh(s^{-\frac{4}{3}})
\end{equation}
with
\begin{equation}\label{def-J01}
J_1(z) = \frac{1}{8 f(z)^{1/2}}N_\alpha(z)\begin{pmatrix}
0 & 0 & 0\\
0 & 1 + 4 \alpha^2 & -2 \ii\\
0 & -2 \ii & -1 - 4\alpha^2
\end{pmatrix}
N_\alpha(z)^{-1},
\end{equation}
which is \eqref{match-P0}.

This completes the proof of Lemma \ref{lem-P0}.
\end{proof}

%We conclude this section by evaluating $E(0)$ for later use.
%\begin{equation}
%\begin{split}
%E(0) =&2^{\frac{\alpha}{3}-\frac{1}{2}} \diag{\left(-2^{\frac{1}{3}}, \frac{1}{3}, -\frac{2^{5/3}}{9}\right)}\begin{pmatrix}
%		1 & 0 & 0\\
%		1 & 0 & -1\\
%		1 & -2 & 1
%	\end{pmatrix}\\
%	& \times\begin{pmatrix}
%	2^{\alpha - \frac{1}{2}} 3^{\frac{3}{2} - \alpha} & 0 & 0\\
%	0 & 0 & \ii 3^{\frac{3}{4} +\frac{\alpha}{2}}c_1^{\alpha-\frac{1}{2}} \pi^{-\frac{1}{2}} s^{-\frac{1}{3}}\\
%	0 &  -\frac{3^{9/4 +\alpha/2}}{2} c_1^{\frac{1}{2}+\alpha} \pi^{\frac{1}{2}} s^{\frac{1}{3}}  & \ii \frac{3^{7/4 +\alpha/2}}{2} c_1^{\alpha-\frac{1}{2}} \pi^{-\frac{1}{2}} s^{-\frac{1}{3}}
%	\end{pmatrix}.
%	\end{split}
%	\end{equation}

%\subsection{Local parametrix near $z=1$}

Similarly, near $z = 1$, we intend to find a function $P^{(1)}(z)$ satisfying the following RH problem.
\begin{rhp}\label{rhp-p1}
\hfill
\begin{itemize}
	\item [\rm (a)] $P^{(1)}(z)$ is analytic in $D(1, \varepsilon) \setminus \Sigma_T$, where $\Sigma_T$ is defined in \eqref{def:sigmaT}.
	\item [\rm (b)] For $z \in D(1, \varepsilon) \cap \Sigma_T$, we have
	\begin{equation}\label{jump-P1}
		P^{(1)}_+(z) = P^{(1)}_-(z) \begin{cases}
			\begin{pmatrix}
	    			0 & 1 & 0\\
	    			-1 & 0 & 0\\
	    			0 & 0 & 1
			\end{pmatrix}, &\quad z \in  D(1, \varepsilon) \cap \Sigma_0^{(1)},\\
			\begin{pmatrix}
	     			1 & 0 & 0\\
	     			e^{s^{2/3}(\lambda_{2}(z)-\lambda_{1}(z))} & 1 & 0\\
	     			0 & 0 & 1
			\end{pmatrix}, &\quad z \in  D(1, \varepsilon) \cap \Sigma_1^{(1)},\\
			\begin{pmatrix}
	     			1 & 0 & 0\\
	     			e^{s^{2/3}(\lambda_{2}(z)-\lambda_{1}(z))} & 1 & 0\\
	     			0 & 0 & 1
			\end{pmatrix}, &\quad z \in  D(1, \varepsilon) \cap \Sigma_5^{(1)}.		
		\end{cases}
	\end{equation}
	\item [\rm (c)] As $s \to \infty$, we have the matching condition
	\begin{equation}\label{match-P1}
		P^{(1)}(z) = \left(I + \Boh(s^{-\frac{2}{3}})\right) N_\alpha(z), \qquad z\in \partial D(1,\varepsilon),
	\end{equation}
where $N_\alpha(z)$ is given in \eqref{def-N}.
\end{itemize}
\end{rhp}

The RH problem \ref{rhp-p1} can be solved with the help of the Bessel parametrix $\Phi_{0}^{(\text{Bes})}$, following the similar spirit in the construction of $P^{(0)}(z)$. The conformal mapping now reads
\begin{equation}\label{def-f}
\tilde{f}(z) = \frac{1}{4} (\lambda_2(z)-\lambda_1(z))^2 = -\tilde{c}_1^2(z-1) - 2 \tilde{c}_1 \tilde{c}_3 (z-1)^2 + \Boh((z-1)^3),\qquad z\to 1,
\end{equation}
where $\tilde{c}_1$ and $\tilde{c}_3$ are defined in \eqref{def-c}; see \eqref{1-lambda1} and \eqref{1-lambda2}.  We now define
\begin{align}\label{def-P1}
P^{(1)}(z) & = \tilde{E}(z) \begin{pmatrix}
	\left(\Phi_{0}^{(\text{Bes})}\right)_{12}(s^{\frac{4}{3}}\tilde{f}(z)) & -\left(\Phi_{0}^{(\text{Bes})}\right)_{11}(s^{\frac{4}{3}}\tilde{f}(z)) & 0\\
	\left(\Phi_{0}^{(\text{Bes})}\right)_{22}(s^{\frac{4}{3}}\tilde{f}(z)) & -\left(\Phi_{0}^{(\text{Bes})}\right)_{21}(s^{\frac{4}{3}}\tilde{f}(z)) & 0\\
	0 & 0 & 1
	\end{pmatrix} \nonumber
   \\
	&~~ \times \diag\left( e^{\frac{s^{1/3}(\lambda_2(z)-\lambda_1(z))}{2}}, e^{-\frac{s^{1/3}(\lambda_2(z)-\lambda_1(z))}{2}},1\right),
	\end{align}
where
\begin{equation}\label{def-E1}
	\tilde{E}(z) = \frac{ N_\alpha(z)}{\sqrt{2}} \begin{pmatrix}
	\pi^{\frac{1}{2}} s^{\frac{1}{3}} \tilde{f}(z)^{\frac{1}{4}} & -\ii \pi^{-\frac{1}{2}} s^{-\frac{1}{3}} \tilde{f}(z)^{-\frac{1}{4}} & 0\\
	\ii \pi^{\frac{1}{2}} s^{\frac{1}{3}} \tilde{f}(z)^{\frac{1}{4}} & -\pi^{-\frac{1}{2}} s^{-\frac{1}{3}} \tilde{f}(z)^{-\frac{1}{4}} & 0\\
	0 & 0 & \sqrt{2}
	\end{pmatrix}
\end{equation}
and $\Phi_{\alpha}^{(\text{Bes})}$ solves the RH problem \ref{rhp:Bessel} with $\alpha=0$.
\begin{lemma}\label{lem-P1}
The matrix-valued function $P^{(1)}(z)$ defined in \eqref{def-P1} solves the RH problem \ref{rhp-p1}.
\end{lemma}
\begin{proof}
By \eqref{def-E1}, it is easily seen that $\tilde{E}(z)$ is analytic in $D(1,\varepsilon)\setminus [1, 1+ \varepsilon)$. For $z \in (1, 1+ \varepsilon)$, it follows from \eqref{jump-N} and \eqref{def-f} that
\begin{align}
\tilde{E}_-(z)^{-1} \tilde{E}_+(z) & = \frac{1}{2} \begin{pmatrix}
	\pi^{-\frac{1}{2}} s^{-\frac{1}{3}} \tilde{f}_-(z)^{-\frac{1}{4}} & -\ii \pi^{-\frac{1}{2}} s^{-\frac{1}{3}} \tilde{f}_-(z)^{-\frac{1}{4}} & 0\\
	\ii \pi^{\frac{1}{2}} s^{\frac{1}{3}} \tilde{f}_-(z)^{\frac{1}{4}} & -\pi^{\frac{1}{2}} s^{\frac{1}{3}} \tilde{f}_-(z)^{\frac{1}{4}} & 0\\
	0 & 0 & \sqrt{2}
	\end{pmatrix}
	\begin{pmatrix}
		0 & 1 & 0\\
		-1 & 0 & 0\\
		0 & 0 & 1
	\end{pmatrix}
\nonumber \\
	&~~ \times
	\begin{pmatrix}
	\pi^{\frac{1}{2}} s^{\frac{1}{3}} \tilde{f}_+(z)^{\frac{1}{4}} & -\ii \pi^{-\frac{1}{2}} s^{-\frac{1}{3}} \tilde{f}_+(z)^{-\frac{1}{4}} & 0\\
	\ii \pi^{\frac{1}{2}} s^{\frac{1}{3}} \tilde{f}_+(z)^{\frac{1}{4}} & -\pi^{-\frac{1}{2}} s^{-\frac{1}{3}} \tilde{f}_+(z)^{-\frac{1}{4}} & 0\\
	0 & 0 & \sqrt{2}
	\end{pmatrix}=I,
	\end{align}
and by \eqref{N1},
\begin{align}\label{E1}
	\tilde{E}(1) & = \frac{C_{\alpha}}{9 \sqrt{2}}\begin{pmatrix}
		-5 \cdot 2^{-\frac 23} & -7 \cdot 2^{\frac 13} & 2^{-\frac 23}\\
		4 & -2 & -2\\
		-2^{\frac 53} & 2^{\frac 83} & -2^{\frac 53}
	\end{pmatrix}\left[3^{\frac{1}{4}} \tilde{c}_1^{\frac{1}{2}} \pi^{\frac{1}{2}} s^{\frac{1}{3}} \begin{pmatrix}
	2 & 0 & 0\\
	-1 & 0 & 0\\
	\frac{1}{2} & 0 & 0
	\end{pmatrix} \right. \nonumber
    \\
	&\left.
    ~~ -\frac{2^{\alpha+1/2}}{\sqrt{3}} \begin{pmatrix}
	0 & 0 & \frac{1}{2}\\
	0 & 0 & 2\\
	0 & 0 & 8
	\end{pmatrix} + \frac{2 \ii}{3^{1/4} \tilde{c}_1^{1/2} \pi^{1/2} s^{1/3}} \begin{pmatrix}
	0 & \alpha - \frac{5}{3} & 0\\
	0 & -\frac{\alpha}{2} - \frac{2}{3} & 0\\
	0 & \frac{\alpha}{4}+ \frac{13}{12} & 0
	\end{pmatrix}\right],
	\end{align}
where $C_\alpha$ and $\tilde{c}_1$ are given in \eqref{def:C-alpha} and \eqref{def-c}, respectively.
We thus conclude that $\tilde{E}(z)$ is analytic in $D(1, \varepsilon)$. Note that
\begin{equation}\label{eq:relation}
\begin{pmatrix}
	\left(\Phi_{0}^{(\text{Bes})}\right)_{12}(s^{\frac{4}{3}}\tilde{f}(z)) & -\left(\Phi_{0}^{(\text{Bes})}\right)_{11}(s^{\frac{4}{3}}\tilde{f}(z))  \\
	\left(\Phi_{0}^{(\text{Bes})}\right)_{22}(s^{\frac{4}{3}}\tilde{f}(z)) & -\left(\Phi_{0}^{(\text{Bes})}\right)_{21}(s^{\frac{4}{3}}\tilde{f}(z))
	\end{pmatrix}= \Phi_{0}^{(\text{Bes})}(s^{\frac{4}{3}}\tilde{f}(z))\sigma_1\sigma_3,
\end{equation}
where the Pauli matrices $\sigma_1$ and $\sigma_3$ are defined in \eqref{def:Pauli}. It is then easy to check that $P^{(1)}(z)$ satisfies the jump condition \eqref{jump-P1} by applying \eqref{jump-Bessel}, item (i) of Proposition \ref{pro-de} and the analyticity of $\tilde{E}(z)$.

Finally, on account of \eqref{def-P1}, \eqref{def-E1}, \eqref{eq:relation} and the asymptotic behavior of the Bessel parametrix $\Phi_{\alpha}^{(\text{Bes})}(z)$ at infinity in \eqref{infty-Bessel}, we obtain after a straightforward computation that, as $s\to \infty$,
\begin{equation}
P^{(1)}(z) N_\alpha(z)^{-1} = I + \frac{\tilde{J}_1(z)}{s^{2/3}}+ \Boh(s^{-\frac{4}{3}}),
\end{equation}
with
\begin{equation}\label{def-J11}
\tilde{J}_1(z) = \frac{1}{8 \tilde{f}(z)^{1/2}}N_\alpha(z)\begin{pmatrix}
-1 & 2 \ii & 0\\
2 \ii & 1  & 0\\
0 & 0& 0
\end{pmatrix}N_\alpha(z)^{-1}.
\end{equation}

This completes the proof of Lemma \ref{lem-P1}.
\end{proof}

For later use, we need to calculate $\tilde{E}'(1)$. The evaluation is
direct and cumbersome by combining \eqref{def-E1} and the asymptotics of $N_\alpha (z)$ and $\tilde f(z)$ near $z=1$
given in \eqref{N1} and \eqref{def-f}. We omit the details but present the result below.
\begin{align}\label{E'1}
	\tilde{E}'(1) & = \frac{C_{\alpha}}{9 \sqrt{2}}\begin{pmatrix}
		-5 \cdot 2^{-\frac 23} & -7 \cdot 2^{\frac 13} & 2^{-\frac 23}\\
		4 & -2 & -2\\
		-2^{\frac 53} & 2^{\frac 83} & -2^{\frac 53}
	\end{pmatrix}
	\left[\frac{3^{1/4} \tilde{c}_1^{1/2} \pi^{1/2} s^{1/3}}{9}\begin{pmatrix}
	\frac{9\tilde{c}_3}{\tilde{c}_1} - \alpha (3\alpha -5) -\frac{1}{6} & 0 & 0\\
	-\frac{9\tilde{c}_3}{2\tilde{c}_1} +\frac{\alpha (3\alpha + 13)}{2} -\frac{35}{12} & 0 & 0\\
	\frac{9\tilde{c}_3}{4\tilde{c}_1} - \frac{\alpha (3\alpha +31)}{4} -\frac{253}{24} & 0 & 0
	\end{pmatrix} \right.
    \nonumber
    \\
	&\qquad+\frac{\ii}{54 \cdot 3^{1/4} \tilde{c}_1^{1/2} \pi^{1/2} s^{1/3}}\begin{pmatrix}
	0 & \frac{27\tilde{c}_3}{\tilde{c}_1}\left(\frac{10}{3} - 2 \alpha\right) - 6 \alpha^3 +15 \alpha + 1 & 0\\
	0 & \frac{27\tilde{c}_3}{\tilde{c}_1}\left(\frac{4}{3} +\alpha\right) +3 \alpha^3 +27 \alpha^2+\frac{39}{2} \alpha -14 & 0\\
	0 & -\frac{27\tilde{c}_3}{\tilde{c}_1}\left(\frac{13}{6} +\frac{\alpha}{2}\right) - \frac{6 \alpha^3 +108 \alpha^2 +417 \alpha + 269}{4} & 0
	\end{pmatrix}
   \nonumber \\
	&\left. \qquad +\frac{2^{\alpha+1/2}\sqrt{3}}{27} \begin{pmatrix}
	0 & 0 & 2\alpha - \frac{8}{3}\\
	0 & 0 & 8\alpha - \frac{14}{3}\\
	0 & 0 & 32 \alpha + \frac{16}{3}
	\end{pmatrix}\right],
	\end{align}
where the matrix $C_\alpha$ and the constants $\tilde c_i$, $i=1,2,3$, are given in \eqref{def:C-alpha} and \eqref{def-c}.

\subsection{Final transformation}
The final transformation is defined by
\begin{equation}\label{def-R}
R(z) = \begin{cases}
S(z) P^{(0)}(z)^{-1}, & \qquad z \in D(0, \varepsilon),\\
S(z) P^{(1)}(z)^{-1}, & \qquad z \in D(1, \varepsilon),\\
S(z) N_\alpha(z)^{-1}, &\qquad \textrm{elsewhere}.
\end{cases}
\end{equation}
It is then easily seen that $R(z)$ satisfies the following RH problem.
\begin{rhp}
\hfill
\begin{itemize}
\item [\rm (a)] $R(z)$ is analytic in $\mathbb{C} \setminus \Sigma_R$;
where the contour $\Sigma_R$ is shown in Figure \ref{figure-R}.

\item [\rm (b)] For $z \in \Sigma_R$, we have
$$R_+(z) = R_-(z) J_R(z),$$
where
\begin{equation}
J_R(z) = \begin{cases}
P^{(0)}(z)N_\alpha(z)^{-1} & \quad z \in \partial D(0, \varepsilon),\\
P^{(1)}(z)N_\alpha(z)^{-1} & \quad z \in \partial D(1, \varepsilon),\\
N_\alpha(z)J_S(z)N_\alpha(z)^{-1} &\quad \textrm{$z\in \Sigma_R \setminus (\partial D(0, \varepsilon)\cup \partial D(1, \varepsilon))$},
\end{cases}
\end{equation}
and where $J_S(z)$ is defined in \eqref{jump-S}.

\item [\rm (c)] As $z \to \infty$, we have
\begin{equation}\label{infty-R}
R(z) = I + \frac{\mathsf{R}_1}{z} + \Boh(z^{-2}),
\end{equation}
where $\mathsf{R}_1$ is independent of $z$.
\end{itemize}
\end{rhp}

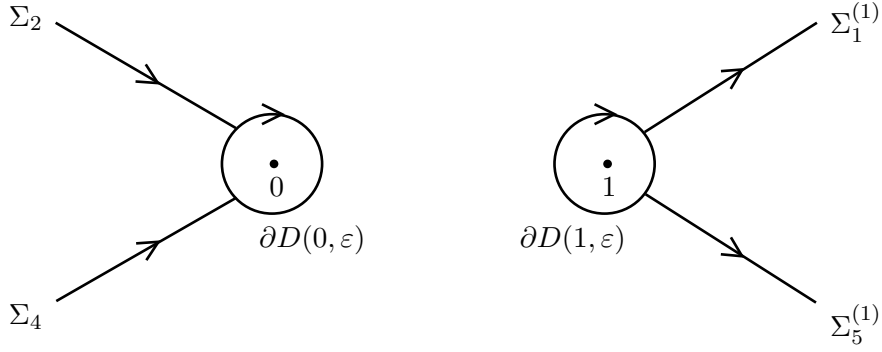
\begin{figure}[ht]
\centering

\tikzset{every picture/.style={line width=1pt}} %set default line width to 0.75pt

\begin{tikzpicture}[x=0.75pt,y=0.75pt,yscale=-1,xscale=1]

%Shape: Circle [id:dp36796173057489023]
\draw  [color={rgb, 255:red, 0; green, 0; blue, 0 }  ,draw opacity=1.5 ] (153,141) .. controls (153,127.19) and (164.19,116) .. (178,116) .. controls (191.81,116) and (203,127.19) .. (203,141) .. controls (203,154.81) and (191.81,166) .. (178,166) .. controls (164.19,166) and (153,154.81) .. (153,141) -- cycle ;
\draw   (173,111) -- (183,116) -- (173,121) ;
%Shape: Circle [id:dp9921168627471101]
\draw  [color={rgb, 255:red, 0; green, 0; blue, 0 }  ,draw opacity=1.5 ] (319,141) .. controls (319,127.19) and (330.19,116) .. (344,116) .. controls (357.81,116) and (369,127.19) .. (369,141) .. controls (369,154.81) and (357.81,166) .. (344,166) .. controls (330.19,166) and (319,154.81) .. (319,141) -- cycle ;
\draw   (339,111) -- (349,116) -- (339,121) ;
%Straight Lines [id:da9433756333418531]
\draw    (70,70) -- (160,123) ;
\draw   (114.97,90.94) -- (121.02,100.35) -- (109.87,99.54) ;

%Straight Lines [id:da9524303591574991]
\draw    (70.31,210) -- (160.31,157.94) ;
\draw   (115.28,189.43) -- (121.32,180.19) -- (110.17,180.98) ;

%Straight Lines [id:da0901474061399622]
\draw    (364,125) -- (450,70) ;
\draw   (406.97,103.27) -- (412.75,93.51) -- (402.09,94.34) ;

%Straight Lines [id:da2746918849479427]
\draw    (364.47,156.25) -- (449.53,210.75) ;
\draw   (407.07,177.94) -- (412.68,187.45) -- (402.04,186.46) ;

%Shape: Circle [id:dp18299509957993376]
\draw  [fill={rgb, 255:red, 0; green, 0; blue, 0 }  ,fill opacity=1 ] (180.5,141) .. controls (180.5,140.17) and (179.83,139.5) .. (179,139.5) .. controls (178.17,139.5) and (177.5,140.17) .. (177.5,141) .. controls (177.5,141.83) and (178.17,142.5) .. (179,142.5) .. controls (179.83,142.5) and (180.5,141.83) .. (180.5,141) -- cycle ;
%Shape: Circle [id:dp8329666821519873]
\draw  [fill={rgb, 255:red, 0; green, 0; blue, 0 }  ,fill opacity=1 ] (347,141) .. controls (347,140.17) and (346.33,139.5) .. (345.5,139.5) .. controls (344.67,139.5) and (344,140.17) .. (344,141) .. controls (344,141.83) and (344.67,142.5) .. (345.5,142.5) .. controls (346.33,142.5) and (347,141.83) .. (347,141) -- cycle ;

% Text Node
\draw (175,146) node [anchor=north west][inner sep=0.75pt]   [align=left] {0};
% Text Node
\draw (341,146) node [anchor=north west][inner sep=0.75pt]   [align=left] {1};
% Text Node
\draw (45,60) node [anchor=north west][inner sep=0.75pt]   [align=left] {$\Sigma_2$};
% Text Node
\draw (45,210) node [anchor=north west][inner sep=0.75pt]   [align=left] {$\Sigma_4$};
% Text Node
\draw (455,58) node [anchor=north west][inner sep=0.75pt]   [align=left] {$\Sigma_1^{(1)}$};
% Text Node
\draw (455,210) node [anchor=north west][inner sep=0.75pt]   [align=left] {$\Sigma_5^{(1)}$};

% Text Node
\draw (170,170) node [anchor=north west][inner sep=0.75pt]   [align=left] {$\partial D(0,\varepsilon)$};
% Text Node
\draw (300,170) node [anchor=north west][inner sep=0.75pt]   [align=left] {$\partial D(1,\varepsilon)$};

\end{tikzpicture}
\caption{The jump contour $\Sigma_R$ for the RH problem for $R$.}
 \label{figure-R}
\end{figure}

On account of Corollary \ref{cor}, the matching conditions \eqref{match-P0} and \eqref{match-P1}, it is readily seen that $J_R(z) \to I$ as $s\to \infty$. By a standard argument (cf. \cite{Deift1999}), we conclude that, as $s\to \infty$,
\begin{equation}\label{def-RR}
R(z) = I + \frac{R_1(z)}{s^{2/3}} + \Boh(s^{-\frac{4}{3}}) \quad \textrm{and} \quad \frac{\ud}{\ud z} R(z) = \frac{R_1'(z)}{s^{2/3}} + \Boh(s^{-\frac{4}{3}}),
\end{equation}
uniformly for $z\in \mathbb{C}\setminus \Sigma_R$. Moreover, the function $R_1(z)$ in \eqref{def-RR} is analytic in $\mathbb{C} \setminus (\partial D(0, \varepsilon) \cup \partial D(1, \varepsilon))$ with asymptotic behavior $\Boh(1/z)$ as $z\to \infty$, and satisfies
$$
R_{1,+}(z)-R_{1,-}(z)= \left\{
                         \begin{array}{ll}
                           J_1(z), & \hbox{$z\in \partial D(0,\varepsilon)$,} \\
                           \tilde J_1(z), & \hbox{$z\in \partial D(1,\varepsilon)$,}
                         \end{array}
                       \right.
$$
where the functions $J_1(z)$ and $\tilde{J}_1(z)$ are given in \eqref{def-J01} and \eqref{def-J11}, respectively.
By Cauchy's residue theorem, we have
\begin{align}\label{def-R1}
R_1(z) & = \frac{1}{2 \pi \ii} \oint_{\partial D(0, \varepsilon)} \frac{J_1(\zeta)}{z - \zeta} \ud \zeta + \frac{1}{2 \pi \ii} \oint_{\partial D(1, \varepsilon)} \frac{\tilde{J}_1(\zeta)}{z - \zeta} \ud \zeta
\nonumber
\\
& = \begin{cases}
\frac{\Res_{\zeta = 0} J_1(\zeta)}{z} + \frac{\Res_{\zeta = 1} \tilde{J}_1(\zeta)}{z-1}, \quad z \in \mathbb{C} \setminus (D(0, \varepsilon) \cup D(1, \varepsilon)),\\
\frac{\Res_{\zeta = 0} J_1(\zeta)}{z} + \frac{\Res_{\zeta = 1} \tilde{J}_1(\zeta)}{z-1}- J_1(z), \quad z \in D(0, \varepsilon),\\
\frac{\Res_{\zeta = 0} J_1(\zeta)}{z} + \frac{\Res_{\zeta = 1} \tilde{J}_1(\zeta)}{z-1} - \tilde{J}_1(z), \quad z \in D(1, \varepsilon).
\end{cases}
\end{align}

We conclude this section with the calculation of $R_1'(1)$. Recall $\tilde{J}_1(z)$ in \eqref{def-J11}, we have from the asymptotics of $N_\alpha(z)$ and $\tilde{f}(z)$ near $z=1$ in \eqref{N1} and \eqref{def-f} that
\begin{equation}\label{J1}
\tilde{J}_1(z) = \frac{\Res_{\zeta = 1} \tilde{J}_1(\zeta)}{z-1} + \mathcal{J}_0 + \mathcal{J}_{1}(z-1) + \Boh ((z-1)^2), \qquad z \to 1,
\end{equation}
where $\mathcal{J}_0$ and $\mathcal{J}_1$ are two constant matrices.
This, together with \eqref{def-R1}, implies that
\begin{equation}\label{R1'}
R_1'(1) = -\mathcal{J}_{1} - \Res_{\zeta = 0} J_1(\zeta).
\end{equation}
Although the explicit formula of $\mathcal{J}_1$ is available, we decide not to include
it due to the complicated form. For the term $\displaystyle \Res_{\zeta = 0} J_1(\zeta)$, combining \eqref{def-J01}, \eqref{N1} and \eqref{def-tildef} together, we have
\begin{align}\label{res-J1}
\Res_{\zeta = 0} J_1(\zeta) & = C_{\alpha}\begin{pmatrix}
		-5 \cdot 2^{-\frac 23} & -7 \cdot 2^{\frac 13} & 2^{-\frac 23}\\
		4 & -2 & -2\\
		-2^{\frac 53} & 2^{\frac 83} & -2^{\frac 53}
	\end{pmatrix} \begin{pmatrix}
	0 & 0 & 0\\
	0 & 0 & 0\\
	0 & \frac{3^{3/2}(4\alpha^2-1)}{16 |c_1|} & 0
	\end{pmatrix}\nonumber\\
	& \qquad	\times \begin{pmatrix}
		-5 \cdot 2^{-\frac 23} & -7 \cdot 2^{\frac 13} & 2^{-\frac 23}\\
		4 & -2 & -2\\
		-2^{\frac 53} & 2^{\frac 83} & -2^{\frac 53}
	\end{pmatrix}^{-1} C_{\alpha}^{-1},
\end{align}
where $C_{\alpha}$ is given in \eqref{def:C-alpha}.

%-------------------------------------------------------------------------------------------------------------------------------------------------
\section{Proof of Theorem \ref{th1}}\label{sec:proof}
We start with derivation of asymptotics of $\frac{\partial}{\partial s} F(s; \rho)$. By \eqref{ds-F} and \eqref{def-T}, it follows that
\begin{equation}\label{eq:FsT}
\frac{\partial}{\partial s} F(s; \rho)  = - \frac{1}{2 \pi \ii} \lim_{z \to s} \left(X(z)^{-1} X'(z) \right)_{21} = - \frac{1}{2 \pi \ii s} \lim_{z \to 1} \left(T(z)^{-1} T'(z) \right)_{21}
\end{equation}
Inverting the transformation $T \to S$ given in \eqref{def-S}, we have
\begin{equation}
T(z)=A S(z)\diag(e^{s^{2/3}\lambda_{1}(z)},e^{s^{2/3}\lambda_{2}(z)},e^{s^{2/3}\lambda_{3}(z)}),
\end{equation}
where $A$ is an invertible matrix that is independent of $z$. Thus,
\begin{align}
\lim_{z \to 1}  \left(T(z)^{-1} T'(z) \right)_{21}
&= \lim_{z \to 1} \left(\diag(e^{-s^{2/3}\lambda_{1}(z)},e^{-s^{2/3}\lambda_{2}(z)},e^{-s^{2/3}\lambda_{3}(z)})S(z)^{-1} S'(z)
\right.
\nonumber
\\
& \left. ~~ \times \diag(e^{s^{2/3}\lambda_{1}(z)},e^{s^{2/3}\lambda_{2}(z)},e^{s^{2/3}\lambda_{3}(z)})\right)_{21}=\lim_{z \to 1}  \left(S(z)^{-1} S'(z) \right)_{21},
\end{align}
where we have made use of \eqref{1-lambda1} and \eqref{1-lambda2} in the second equality. By further tracking back the transformation $S \to R$ given in \eqref{def-R}, we obtain from \eqref{eq:FsT} and the above formula that
\begin{align}\label{ds-F1}
\frac{\partial}{\partial s} F(s; \rho) &= - \frac{1}{2 \pi \ii s} \lim_{z \to 1}\left(S(z)^{-1} S'(z) \right)_{21}
\nonumber
\\
& = - \frac{1}{2 \pi \ii s} \lim_{z \to 1} \left(P^{(1)}(z)^{-1} R(z)^{-1} R'(z) P^{(1)}(z) + P^{(1)}(z)^{-1} (P^{(1)}(z))' \right)_{21}.
\end{align}
This, together with explicit expression of $P^{(1)}(z)$ in \eqref{def-P1}, estimates of $R(z)$, $R'(z)$ in \eqref{def-RR} and the local behaviors of $\lambda_i(z)$, $i=1,2$, near $z=1$ in \eqref{1-lambda1} and \eqref{1-lambda2}, implies that
\begin{align}\label{F-expansion}
&\frac{\partial}{\partial s} F(s; \rho)
\nonumber
\\ & = - \frac{1}{2 \pi \ii s} \lim_{z \to 1} \Bigg(B(s^{\frac{4}{3}} \tilde{f}(z))^{-1} \tilde{E}(z)^{-1} \left(\frac{R_1'(z)}{s^{2/3}} + \Boh(s^{-\frac{4}{3}})\right) \tilde{E}(z) B(s^{\frac{4}{3}} \tilde{f}(z))
\nonumber
\\
&~~ + B(s^{\frac{4}{3}} \tilde{f}(z))^{-1} \tilde{E}(z)^{-1}\tilde{E}'(z)B(s^{\frac{4}{3}} \tilde{f}(z)) + s^{\frac{4}{3}} \tilde{f}'(z)B(s^{\frac{4}{3}} \tilde{f}(z))^{-1}B'(s^{\frac{4}{3}} \tilde{f}(z)) \Bigg)_{21},
\end{align}
where
\begin{equation}
B(z): = \begin{pmatrix}
	\left(\Phi_{0}^{(\text{Bes})}\right)_{12}(z) & -\left(\Phi_{0}^{(\text{Bes})}\right)_{11}(z) & 0\\
	\left(\Phi_{0}^{(\text{Bes})}\right)_{22}(z) & -\left(\Phi_{0}^{(\text{Bes})}\right)_{21}(z) & 0\\
	0 & 0 & 1
	\end{pmatrix}.
\end{equation}
%
%To prove the large gap asymptotics $F(s; \rho)$ as $s \to \infty$, let us recall the differential identities established in Proposition \ref{pro-de}. First, we derive the. Tracing back the invertible transformations $X \to T \to S \to R$ gives us
%
% Recalling the and the definition of $R(z)$ in \eqref{def-RR}, the above formula can be rewritten as

We next calculate the three terms on the right hand side of \eqref{F-expansion} one by one. To proceed, we observe from
\eqref{def-bessel} and properties of the modified Bessel functions $I_0$ and $K_0$ given in \cite[Chapter 10]{NIST} that,
as $z \to 0$,
\begin{equation}
B(z) = \begin{pmatrix}
	1 +\Boh(z) & \Boh(\ln{z}) & 0\\
	\frac{\pi \ii}{2} z + \Boh(z^2) & -1 + \Boh(z \ln{z}) & 0\\
	0 & 0 & 1
	\end{pmatrix},
\end{equation}
and
\begin{equation}
B(z)^{-1} =  \begin{pmatrix}
	\left(\Phi_{0}^{(\text{Bes})}\right)_{21}(z) & -\left(\Phi_{0}^{(\text{Bes})}\right)_{11}(z) & 0\\
	\left(\Phi_{0}^{(\text{Bes})}\right)_{22}(z) & -\left(\Phi_{0}^{(\text{Bes})}\right)_{12}(z) & 0\\
	0 & 0 & 1
	\end{pmatrix}
  =\begin{pmatrix}
	1 +\Boh(z\ln z) & \Boh(\ln{z}) & 0\\
	\frac{\pi \ii}{2} z + \Boh(z^2) & -1 + \Boh(z) & 0\\
	0 & 0 & 1
	\end{pmatrix}.
\end{equation}
A combination of the above two formulas gives us
\begin{equation}\label{BB'}
\lim_{z \to 0} \left(B(z)^{-1}B'(z)\right)_{21} = -\frac{\pi \ii}{2}.
\end{equation}
In addition, it is straightforward to check that for any $3 \times 3$ matrix $M$, one has
\begin{equation}\label{BMB}
\lim_{z \to 0} \left(B^{-1}(z)MB(z)\right)_{21} = -(M)_{21}.
\end{equation}
For the first term in \eqref{F-expansion}, we obtain from $\tilde{E}(1)$ in \eqref{E1} and $R_1'(1)$ in \eqref{R1'} that
\begin{equation}
\left(s^{-\frac{2}{3}}\tilde{E}(1)^{-1}R_1'(1) \tilde{E}(1)\right)_{21} = -\frac{\ii (4\alpha^2 -1) \tilde{c}_1 \pi}{24|c_1|} - \frac{3 \ii \tilde{c}_3 \pi}{8\tilde{c}_1} - \frac{\ii (8\alpha^2+1) \pi}{48},
\end{equation}
where $\tilde{c}_i, i =1, 2, 3$ are given in \eqref{def-c} and $c_1$ is given in \eqref{tildec}. This, together with \eqref{BMB} and \eqref{def-f}, implies that
\begin{align}\label{term-1}
 & \lim_{z \to 1} \left(B(s^{\frac{4}{3}} \tilde{f}(z))^{-1} \tilde{E}(z)^{-1} \left(\frac{R_1'(z)}{s^{2/3}} + \Boh(s^{-\frac{4}{3}})\right) \tilde{E}(z) B(s^{\frac{4}{3}} \tilde{f}(z))\right)_{21}
\nonumber
 \\
&=- \lim_{z \to 1}\left(\tilde{E}(z)^{-1} \left(\frac{R_1'(z)}{s^{2/3}} + \Boh(s^{-\frac{4}{3}})\right) \tilde{E}(z) \right)_{21}
\nonumber
 \\
 &= \frac{\ii (4\alpha^2 -1) \tilde{c}_1 \pi}{24|c_1|}  + \frac{3 \ii \tilde{c}_3 \pi}{8\tilde{c}_1} + \frac{\ii (8\alpha^2+1) \pi}{48} + \Boh(s^{-\frac{2}{3}}).
\end{align}
Similarly, with the aids of $\tilde{E}(1)$ and $\tilde{E}'(1)$ in \eqref{E1} and \eqref{E'1}, we have
\begin{multline}\label{term-2}
\lim_{z \to 1} \left(B(s^{\frac{4}{3}} \tilde{f}(z))^{-1} \tilde{E}(z)^{-1}\tilde{E}'(z)B(s^{\frac{4}{3}} \tilde{f}(z)) \right)_{21} =
- \lim_{z \to 1}\left(\tilde{E}(z)^{-1}\tilde{E}'(z) \right)_{21}
\\
 = -\frac{\ii \sqrt{3} \alpha \pi \tilde{c}_1}{3} s^{\frac 23}.
\end{multline}
The third term in \eqref{F-expansion} can be evaluated directly by applying \eqref{def-f} and \eqref{BB'}, which gives
\begin{equation}\label{term-3}
\lim_{z \to 1} \left(s^{\frac{4}{3}} \tilde{f}'(z)B^{-1}(s^{\frac{4}{3}} \tilde{f}(z)) B'(s^{\frac{4}{3}} \tilde{f}(z)) \right)_{21} = \frac{\ii \pi \tilde{c}_1^2}{2}s^{\frac{4}{3}}.
\end{equation}

Finally, substituting \eqref{term-1}--\eqref{term-3} and \eqref{def-c} into \eqref{F-expansion}, we obtain
\begin{equation*}
\frac{\partial}{\partial s} F(s; \rho) = -\frac{3}{2^{8/3}}s^{\frac{1}{3}} + \frac{\rho}{2} + \frac{3\alpha - \rho^2}{3 \cdot 2^{4/3}} s^{-\frac{1}{3}} -\frac{\alpha \rho}{3 \cdot 2^{2/3}} s^{-\frac{2}{3}} -\frac{12\alpha^2+1}{72} s^{-1} + \Boh(s^{-\frac{4}{3}}),
\end{equation*}
as $s \to \infty$. Integrating the above formula gives us
\begin{equation}\label{F2}
F(s; \rho) = -\frac{9}{16 \cdot 2^{2/3}}s^{\frac{4}{3}} + \frac{\rho}{2}s + \frac{3\alpha - \rho^2}{2^{7/3}} s^{\frac{2}{3}} - \frac{\alpha \rho}{ 2^{2/3}} s^{\frac{1}{3}}-\frac{12\alpha^2+1}{72} \ln{s}
+ C(\rho)+ \Boh(s^{-\frac{1}{3}}),
\end{equation}
uniformly for $\rho$ in any compact subset of $\mathbb{R}$, where $C(\rho)$ is a constant that might be dependent on the parameters $\alpha$ and $\rho$.

To find more information about $C(\rho)$, we come to $\frac{\partial}{\partial \rho} F(s; \rho)$. From \eqref{drho-F} and \eqref{def-S1}, we
have
\begin{equation}\label{eq:Fparrho}
\frac{\partial}{\partial \rho} F(s; \rho) = - (X_1)_{31}+ \frac{\rho (\rho^2 + 9\alpha)}{27}= -s^{\frac{1}{3}}(S_1)_{31} - s d_1 + \frac{\rho (\rho^2 + 9\alpha)}{27},
\end{equation}
where $S_1$ and $d_1$ are given in \eqref{def-S1} and \eqref{d12}. Recall that
\begin{equation}
S(z) = R(z)N_\alpha(z), \qquad z \in \mathbb{C} \setminus (D(0, \varepsilon) \cup D(1, \varepsilon) \cup \Sigma_T),
\end{equation}
it follows \eqref{infty-S}, \eqref{infty-N1} and \eqref{infty-R} that
\begin{equation}
S_1 = \mathsf{N}_1 + \mathsf{R}_1
\end{equation}
where $\mathsf{N}_1$ and $\mathsf{R}_1$ are the coefficients of $1/z$ for $R(z)$ and $N(z)$ at infinity given in \eqref{infty-N1} and \eqref{infty-R}. It is clear from \eqref{def-RR} that
$
\mathsf{R}_1 = \Boh(s^{-2/3}).
$
This, together with \eqref{def-N1}, implies that
\begin{equation}
(S_1)_{31} = (\mathsf{N}_1 + \mathsf{R}_1)_{31}=\frac{\alpha}{2^{2/3}}+\Boh(s^{-\frac 23}).
\end{equation}
We then obtain from \eqref{eq:Fparrho}, \eqref{d12} and the above formula that
\begin{equation}
\frac{\partial}{\partial \rho} F(s; \rho) = \frac{s}{2} - \frac{\rho}{2^{4/3}} s^{\frac{2}{3}} - \frac{\alpha}{2^{2/3}} s^{\frac{1}{3}} + \frac{\rho (\rho^2 + 9\alpha)}{27}+ \Boh(s^{-\frac{1}{3}}), \qquad s\to \infty.
\end{equation}
Comparing this approximation with the asymptotics of $F (s; \rho)$ given in \eqref{F2}, it is easily seen that
\begin{equation}
C'(\rho) = \frac{\rho (\rho^2 + 9\alpha)}{27}.
\end{equation}
Hence,
\begin{equation}\label{c-rho}
C(\rho) = \frac{\rho^4}{108}+\frac{\alpha \rho^2}{6} + C,
\end{equation}
where $C$ is an undetermined constant independent of $s$ and $\rho$. Inserting \eqref{c-rho} into \eqref{F2} leads to our final asymptotic result \eqref{F1}.

This completes the proof of Theorem \ref{th1}.
\qed

\appendix
\section{The Bessel parametrix}
\label{A}
The Bessel parametrix $\Phi_{\alpha}^{(\text{Bes})}(z)$, which depends on a parameter $\alpha>-1$,
is the unique solution of the following RH problem.

\begin{rhp}\label{rhp:Bessel}
\hfill
\begin{itemize}
	\item [\rm(a)] $\Phi_{\alpha}^{(\text{Bes})}(z)$ is defined and analytic in $\mathbb{C} \setminus (\cup_{j=1}^3 \Gamma_j \cup\{0\})$, where the contours $\Gamma_j$, $j=1,2,3$, are shown in Fig. \ref{figure-Bessel}.
	\item [\rm(b)] For $z \in \cup_{j=1}^3 \Gamma_j$, we have
	\begin{equation}\label{jump-Bessel}
		\Phi_{\alpha,+}^{(\text{Bes})}(z) = \Phi_{\alpha, -}^{(\text{Bes})}(z) \begin{cases}
		\begin{pmatrix}
			1 & e^{\alpha \pi \ii}\\
			0 & 1
		\end{pmatrix}, \quad & z \in \Gamma_1,\\
		\begin{pmatrix}
			0 &-1\\
			1 & 0
		\end{pmatrix}, \quad & z \in \Gamma_2,\\
		\begin{pmatrix}
			1 & e^{-\alpha \pi \ii}\\
			0 & 1
		\end{pmatrix}, \quad & z \in \Gamma_3.
		\end{cases}
	\end{equation}
	\item [\rm(c)] As $z \to \infty$, we have
	\begin{multline}\label{infty-Bessel}
	\Phi_{\alpha}^{(\text{Bes})}(z) = \frac{(\pi^2 z)^{-\sigma_3/4}}{\sqrt{2}} \begin{pmatrix}
		\ii & 1\\
		1 & \ii
	\end{pmatrix}
\\ \times \left(I + \frac{1}{8z^{1/2}}\begin{pmatrix}
	1+4\alpha^2 & -2 \ii\\
	-2 \ii & -1-4\alpha^2
	\end{pmatrix} + \Boh\left(\frac{1}{z}\right) \right)e^{-z^{1/2}\sigma_3},
	\end{multline}
where $\sigma_3$ is defined in \eqref{def:Pauli}.
	\item [\rm(d)] As $z \to 0$, we have
	\begin{equation}
		\Phi_{\alpha}^{(\text{Bes})}(z) = \begin{cases}
		\Boh\begin{pmatrix}
		|z|^{\frac{\alpha}{2}} & |z|^{\frac{\alpha}{2}}\\
		|z|^{\frac{\alpha}{2}} & |z|^{\frac{\alpha}{2}}
		\end{pmatrix}, \quad & \alpha <0,\\
		\Boh\begin{pmatrix}
		\ln{|z|} & \ln{|z|}\\
		\ln{|z|} & \ln{|z|}
		\end{pmatrix}, \quad & \alpha =0,\\
		\Boh\begin{pmatrix}
		|z|^{-\frac{\alpha}{2}} & |z|^{\frac{\alpha}{2}}\\
		|z|^{-\frac{\alpha}{2}} & |z|^{\frac{\alpha}{2}}
		\end{pmatrix}, \quad & \textrm{$\alpha >0$ and $z \in \rm{I}$},\\
		\Boh\begin{pmatrix}
		|z|^{-\frac{\alpha}{2}} & |z|^{-\frac{\alpha}{2}}\\
		|z|^{-\frac{\alpha}{2}} & |z|^{-\frac{\alpha}{2}}
		\end{pmatrix}, \quad & \textrm{$\alpha >0$ and $z \in \rm{II} \cup \rm{III}$},
		\end{cases}
	\end{equation}
where the domains $\rm{I}$--$\rm{III}$ are illustrated in Figure \ref{figure-Bessel}.
\end{itemize}
\end{rhp}

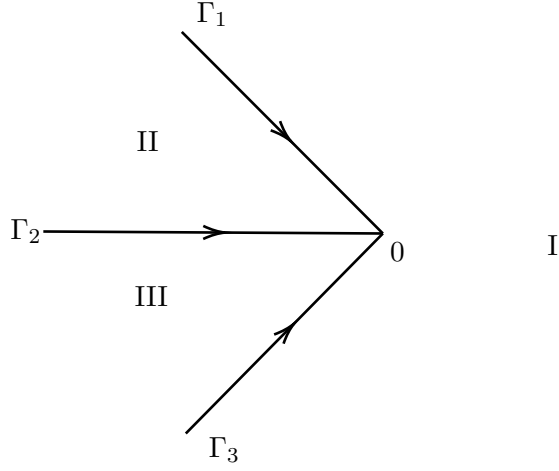
\begin{figure}[t]
\centering

\tikzset{every picture/.style={line width=1pt}} %set default line width to 0.75pt

\begin{tikzpicture}[x=0.75pt,y=0.75pt,yscale=-1,xscale=1]

%Straight Lines [id:da8786137432471083]
\draw    (100,130) -- (269.5,131) ;
\draw [shift={(190.75,130.54)}, rotate = 180.34] [color={rgb, 255:red, 0; green, 0; blue, 0 }  ][line width=1]    (10.93,-3.29) .. controls (6.95,-1.4) and (3.31,-0.3) .. (0,0) .. controls (3.31,0.3) and (6.95,1.4) .. (10.93,3.29)   ;
%Straight Lines [id:da604422146846848]
\draw    (169.12,29.61) -- (269.5,131) ;
\draw [shift={(223.53,84.57)}, rotate = 225.29] [color={rgb, 255:red, 0; green, 0; blue, 0 }  ][line width=1]    (10.93,-3.29) .. controls (6.95,-1.4) and (3.31,-0.3) .. (0,0) .. controls (3.31,0.3) and (6.95,1.4) .. (10.93,3.29)   ;
%Straight Lines [id:da7733071420050464]
\draw    (269.5,131) -- (171.12,231.61) ;
\draw [shift={(225.2,176.3)}, rotate = 134.36] [color={rgb, 255:red, 0; green, 0; blue, 0 }  ][line width=1]    (10.93,-3.29) .. controls (6.95,-1.4) and (3.31,-0.3) .. (0,0) .. controls (3.31,0.3) and (6.95,1.4) .. (10.93,3.29)   ;

% Text Node
\draw (350,130) node [anchor=north west][inner sep=0.75pt]   [align=left] {I};
% Text Node
\draw (271.5,134) node [anchor=north west][inner sep=0.75pt]   [align=left] {0};
% Text Node
\draw (145,78) node [anchor=north west][inner sep=0.75pt]   [align=left] {II};
% Text Node
\draw (144,156) node [anchor=north west][inner sep=0.75pt]   [align=left] {III};
% Text Node
\draw (175,13) node [anchor=north west][inner sep=0.75pt]   [align=left] {$\Gamma_1$};
% Text Node
\draw (82,122) node [anchor=north west][inner sep=0.75pt]   [align=left] {$\Gamma_2$};
% Text Node
\draw (181,232) node [anchor=north west][inner sep=0.75pt]   [align=left] {$\Gamma_3$};

\end{tikzpicture}
\caption{The jump contours $\Gamma_j$, $j=1,2,3$, and the domains I--III in the RH problem for $\Phi_{\alpha}^{(\text{Bes})}$.}
  \label{figure-Bessel}
\end{figure}

Although the above model RH problem is slightly different from the standard Bessel parametrix introduced in \cite{Kuijlaars2004}, they are actually equivalent. From \cite{Kuijlaars2004}, we have
\begin{equation}\label{def-bessel}
	\Phi_{\alpha}^{(\text{Bes})}(z) = \begin{pmatrix}
	I_{\alpha}(z^{1/2})& \frac{\ii}{\pi} K_{\alpha}(z^{1/2})\\
	\pi \ii z^{1/2}I_{\alpha}'(z^{1/2}) & -z^{1/2}K_{\alpha}'(z^{1/2})
	\end{pmatrix}\begin{cases}
	\begin{pmatrix}
    0 & 1
   \\
    1 & 0
    \end{pmatrix}, \quad & z \in \rm{I},\\
	\begin{pmatrix}
	0 & 1\\
	1 & -e^{\alpha \pi \ii}
	\end{pmatrix}, \quad & z \in \rm{II},\\
	\begin{pmatrix}
	0 & 1\\
	1 & e^{-\alpha \pi \ii}
	\end{pmatrix}, \quad & z \in \rm{III},
	\end{cases}
\end{equation}
where $I_{\alpha}(z)$ and $K_{\alpha}(z)$ denote the modified Bessel functions \cite{NIST} and the principal branch is taken for $z^{1/2}$.

\section*{Acknowledgements}
This work was partially supported by National Natural Science Foundation of China under grant numbers 12271105 and 11822104, ``Shuguang Program'' supported by Shanghai Education Development Foundation and Shanghai Municipal Education Commission.

\end{document}